\documentclass[a4paper,10pt]{amsart}
\usepackage[utf8]{inputenc}

\usepackage[
  margin=30mm,
  marginparwidth=25mm,     
  marginparsep=2mm,       
  bottom=25mm,
  ]{geometry}

\usepackage[bbgreekl]{mathbbol}
\usepackage{amsfonts}
\usepackage{latexsym,amssymb,amsthm,mathrsfs,amsmath,amscd,enumerate,enumitem, amscd,color}
\usepackage{mathrsfs}
\usepackage{tikz}
\usepackage{hyperref}
\usetikzlibrary{plotmarks}
\usepackage{cite}
\usepackage{soul}

\DeclareSymbolFontAlphabet{\mathbb}{AMSb}
\DeclareSymbolFontAlphabet{\mathbbl}{bbold}
\DeclareSymbolFontAlphabet{\mathbb}{AMSb}
\DeclareSymbolFontAlphabet{\mathbbl}{bbold}
\DeclareFontEncoding{FMS}{}{}
\DeclareFontSubstitution{FMS}{futm}{m}{n}
\DeclareFontEncoding{FMX}{}{}
\DeclareFontSubstitution{FMX}{futm}{m}{n}
\DeclareSymbolFont{fouriersymbols}{FMS}{futm}{m}{n}
\DeclareSymbolFont{fourierlargesymbols}{FMX}{futm}{m}{n}
\DeclareMathDelimiter{\VERT}{\mathord}{fouriersymbols}{152}{fourierlargesymbols}{147}

\newtheorem{thm}{Theorem}[section]

 \newtheorem{prop}[thm]{Proposition}
\theoremstyle{definition}

 \theoremstyle{remark}
 \newtheorem{rem}[thm]{Remark}

\usepackage{hyperref}

\newcommand{\supp}{\mathop{\mathrm{supp}}}

\numberwithin{equation}{section}
\allowdisplaybreaks

\begin{document}


\title[]
 {$L^p$-boundedness properties for some harmonic analysis operators defined by resolvents for a Laplacian with drift in Euclidean spaces}

\author[J.J. Betancor]{Jorge J. Betancor}
\author[J.C. Fari\~na]{Juan C. Fari\~na}
\author[L. Rodr\'{\i}guez-Mesa]{Lourdes Rodr\'{\i}guez-Mesa}
\address{Jorge J. Betancor, Juan C. Fari\~na, Lourdes Rodr\'{\i}guez-Mesa\newline
	Departamento de An\'alisis Matem\'atico, Universidad de La Laguna,\newline
	Campus de Anchieta, Avda. Astrof\'isico S\'anchez, s/n,\newline
	38721 La Laguna (Sta. Cruz de Tenerife), Spain}
\email{jbetanco@ull.es, jcfarina@ull.edu.es, lrguez@ull.edu.es
}

\thanks{The authors are partially supported by grant PID2019-106093GB-I00 from the Spanish Government}



\date{\today}


\begin{abstract}
We consider the Laplacian with drift in $\mathbb R^n$ defined by $\Delta_\nu = \sum_{i=1}^n(\frac{\partial^2}{\partial x_i^2} + 2 \nu_i\frac{\partial }{\partial{x_i}})$ where $\nu=(\nu_1,\ldots,\nu_n)\in \mathbb R^n\setminus\{0\}$. The operator $\Delta_\nu$ is selfadjoint with respect to the measure $d\mu_\nu(x)=e^{2\langle\nu,x\rangle}dx$. This measure is not doubling but it is locally doubling in $\mathbb R^n$. We define, for every $M>0$ and $k \in \mathbb N$, the operators
$$
W^k_{\nu,M,*}(f) = \sup_{t>0}\left|A^k_{\nu,M,t}(f)\right|,\hspace{5mm}g_{\nu,M}^k(f) = \left(\int_0^\infty\left|A^k_{\nu,M,t}(f)\right|^2\frac{dt}{t}\right)^{\frac{1}{2}},\,k\geq 1,
$$
the $\rho$-variation operator
$$
V_\rho\left( \{A^k_{\nu,M,t}\}_{t>0}\right)(f)= \sup_{0<t_1<\cdots<t_\ell,\,\ell \in \mathbb N}\left(\sum^{\ell-1}_{j=1}\left|A^k_{\nu,M,t_j}(f)- A^k_{\nu,M,t_{j+1}}(f)\right|^\rho\right)^{\frac{1}{\rho}},\;\; \rho>2,
$$
and, if $\{t_j\}_{j\in \mathbb N}$ is a decreasing sequence in $(0,\infty)$, the oscillation operator
$$
O(\{A_{\nu,M,t}^k\}_{t>0},\{t_j\}_{j\in \mathbb N})(f)=\Big(\sum_{j\in \mathbb N}\;\;\sup_{t_{j+1}\leq \varepsilon <\varepsilon '\leq t_j}|A^k_{\nu,M,\varepsilon}(f)-A^k_{\nu,M,\varepsilon '}(f)|^2 \Big)^{1/2}.
$$
where $A^k_{\nu,M,t}=t^k\partial^k_t(I-t\Delta_\nu)^{-M}$, $t>0$. We denote by $T_{\nu,M}^k$ any of the above operators. We analyze the boundedness of $T^k_{\nu,M}$ on $L^p(\mathbb R^n,\mu_\nu)$ into itself, for every $1<p<\infty$, and from $L^1(\mathbb R^n,\mu_\nu)$ into $L^{1,\infty}(\mathbb R^n,\mu_\nu)$. 
In addition, we obtain boundedness properties for the operator $G_{\nu,M}^{k,\ell}$, $1\leq \ell <2M$, defined by
$$
G_{\nu,M}^{k,\ell}(f)=\left(\int_0^\infty\left|t^{\ell /2+k}\partial _t^kD^{(\ell)}(I-t\Delta _\nu)^{-M}(f) \right|^2\frac{dt}{t}\right)^{\frac{1}{2}},
$$
for certain differentiation operator $D^{(\ell)}$.

\end{abstract}

\maketitle

\baselineskip=15pt
\section{Introduction}

Let $\nu=(\nu_1,\ldots,\nu_n) \in \mathbb R^n\setminus\{0\}$ and denote by $\mu_\nu$ the measure given by $d\mu_\nu(x) = e^{2\langle \nu,x\rangle}dx$, where as usual $\langle \cdot,\cdot\rangle$ represents the inner product in $\mathbb R^n$. If $\nu_1=1$ and $\nu_j=0$, $j=2,\ldots,n$, then, for any ball $B(x,r)= \{y \in \mathbb R^n: |x-y|<r\}$, $x \in \mathbb R^n$ and $r>0$, 
$$
\displaystyle\mu_\nu(B(x,r))\sim\left\{\begin{array}{ll}
                            r^ne^{2x_1},& \;0<r\leq 1, \\[0.1cm]
                            r^{\frac{n-1}{2}}e^{2(x_1+r)},& \; r>1,
                            \end{array}\right.
$$
(see \cite[(2.10)]{LSjW}). Thus, the measure $\mu_\nu$ is a measure with exponential growth with respect to the Euclidean distance that is not doubling, but which is locally doubling. 

Consider the Laplacian with drift in $\mathbb R^n$ defined by 
$$
\Delta_\nu=\sum^n_{i=1}\Big(\frac{\partial^2}{\partial x_i^2}+2\nu_i\frac{\partial }{\partial{x_i} }\Big).
$$ 
This operator admits a selfadjoint extension in $L^2(\mathbb R^n,\mu_\nu)$ being $-\Delta_\nu$ a positive operator. Moreover, $\Delta_\nu$ generates a semigroup of operators $\{W_t^\nu\}_{t>0}$ in $L^2(\mathbb R^n,\mu_\nu)$. Actually,  $\{W_t^\nu\}_{t>0}$ is a symmetric diffusion semigroup in $L^p(\mathbb R^n,\mu_\nu)$, $1 \leq p < \infty$, in the sense of Stein (\!\! \cite{StLP}) (see \cite[Theorem 11.8]{Gr}). 

Harmonic analysis operators associated with $\Delta_\nu$ have been recently studied. We mention the paper of Lohou\'e and Mustapha \cite{LM} where it is established that the Riesz transforms of any order are bounded on $L^p(\mathbb R^n,\mu_\nu)$, $1<p<\infty$. Their results actually hold in a more general setting. In \cite[Theorem 1]{LSjW} Li, Sj\"ogren, and Wu established that the first order Riesz transform is bounded from $L^1(\mathbb R^n, \mu_\nu)$ into $L^{1,\infty}(\mathbb R^n,\mu_\nu)$. The extension of this result to higher order Riesz transform was obtained in \cite[Theorem 1.1]{LSj4}. According to \cite[Theorem p. 73]{StLP} the maximal operator associated with the semigroup $\{W_t^\nu\}_{t>0}$ is bounded on $L^p(\mathbb R^n,\mu_\nu)$, for every $1<p<\infty$.
In \cite[Theorem 2]{LSjW} it was proved that this maximal operator is bounded from $L^1(\mathbb R^n,\mu_\nu)$ into $L^{1,\infty}(\mathbb R^n,\mu_\nu)$. The behavior on $L^1(\mathbb R^n,\mu_\nu)$ of Littlewood-Paley functions defined by the semigroup $\{W_t^\nu\}_{t>0}$ and the subordinated Poisson semigroup associated with $\{W_t^\nu\}_{t>0}$ were established in \cite[Theorems 1.2 and 1.3]{LSj4}. Since the measure $\mu_\nu$ is of exponential growth, the proofs of the results are more subtle that the corresponding ones in the doubling setting.

Our objective in this paper is to establish $L^p$-boundedness properties for certain maximal operators, Littlewood-Paley functions and variation and oscillation operators involving the resolvent of the operator $\Delta_\nu$. Specifically, we deal with the following operators.

Let $k \in \mathbb N$ and $M>0$. Define, for every $t>0$,
$A^k_{\nu,M,t}:= t^k\partial^k_t(I-t\Delta_\nu)^{-M}$. Consider the maximal operator $W_{\nu,M,*}^k$ given by
$$
W_{\nu,M,*}^k(f) =\sup_{t>0}\left|A^k_{\nu,M,t}(f)\right|,
$$
and define, when $k\geq 1$, the Littlewood-Paley function $g_{\nu,M}^k$ by
$$
g^k_{\nu,M}(f)=\left(\int_0^\infty \left|A^k_{\nu,M,t}(f)\right|^2 \frac{dt}{t}\right)^{\frac{1}{2}}.
$$
Littlewood- Paley functions involving resolvents have been used in the study of functional calculus for perturbated Hodge-Dirac operators (\!\! \cite{AKM}, \cite{BS}, \cite{FMP}, \cite{HM} and \cite{HMP}).

As it is well-known $L^p$- boundedness properties for maximal operators allow us to obtain pointwise convergence of the family defining the maximal operator. Another tool to measure the speed of that convergence is the use of $\rho$-variation and oscillation operators.

Let $\rho >2$. The linear space $E_\rho$ consists of all those $g:(0,\infty) \longrightarrow \mathbb C$ such that 
$$
\|g\|_{E_\rho} = \sup_{0<t_1<\cdots <t_\ell,\, \ell\in \mathbb N}\Big(\sum_{j=1}^{\ell -1} |g(t_j)-g(t_{j+1})|^\rho\Big)^{1/\rho}< \infty.
$$
By identifying the functions that differ in a constant $(E_\rho,\|\cdot\|_{E_\rho})$ is a Banach space.

It is usual to consider the oscillation operator as a substitute of the $\rho$-variation operator when $\rho =2$. Let $\{t_j\}_{j\in \mathbb N}$ be a decreasing sequence in $(0,\infty )$. We say that a  function $g:(0,\infty)\longrightarrow \mathbb C$ is in $O(\{t_j\}_{j\in \mathbb N})$ when $\|g\|_{O(\{t_j\}_{j\in \mathbb N})}<\infty$, where
$$
\|g\|_{O(\{t_j\}_{j\in \mathbb N})}=\Big(\sum_{j\in \mathbb N}\sup_{t_{j+1}\leq \varepsilon _{j+1}<\varepsilon _j\leq t_j}|g(\varepsilon _j)-g(\varepsilon _{j+1})|^2\Big)^{1/2}.
$$ 
By identifying those functions that differ by a constant, $(O(\{t_j\}_{j\in \mathbb N}), \|\cdot \|_{O(\{t_j\}_{j\in \mathbb N})})$ is a Banach space.

Suppose that $\{T_t\}_{t>0}$ is a family of operators defined, for instance, in $L^p(\Omega,\lambda)$, for some $1\leq p <\infty$ and some measure space $(\Omega,\lambda)$. We define the $\rho$-variation operator $V_\rho\left(\{T_t\}_{t>0}\right)$ as follows
$$
V_\rho\left(\{T_t\}_{t>0}\right)(f)(x) =\|g_x\|_{E_\rho},\;\;x \in \Omega,
$$
where, for $x\in \Omega$, $g_x(t) =T_t(f)(x)$, $t>0$.

L\'epingle (\!\! \cite{Le}) established a variational inequality for martingales. Bourgain \cite{Bou} proved that the $\rho$-variation operator defined by ergodic means are bounded from $L^2(\Omega,\lambda)$ into itself. This property was extended to $L^p(\Omega,\lambda)$, with $1<p<\infty$, by Jones, Kaufman, Rosenblatt and Wierdl (\!\! \cite{JKRW}). After \cite{Bou}, variational type inequalities have been extensively studied in probability theory, ergodic theory and harmonic analysis (related to singular integrals and semigroups of operators) (see for instance, \cite{CJRW1}, \cite{CJRW2}, \cite{CMMTV}, \cite{HMMT}, \cite{JKRW, JR, JSW, JW}, \cite{LeMX2}, \cite{MTX1, MTX2, OSTTW, PX} and \cite{TZ}). In order to obtain variational inequalities, in general, it is necessary to assume that $\rho >2$ (see\cite{Qi}).

Consider the $\rho$-variation operator for $\{A_{\nu,M,t}^k\}_{t>0}$ given by 
$$
V_\rho (\{A_{\nu,M,t}^k\}_{t>0})(f)=\sup_{0<t_1<\cdots <t_\ell ,\,\ell \in \mathbb N}\left(\sum_{j=1}^{\ell -1}|A_{\nu, M,t_j}^k(f)-A_{\nu,M,t_{j+1}}^k(f)|^\rho \right)^{1/\rho }.
$$

Some comment about the measurability of the function $V_\rho (\{A_{\nu,M,t}^k\}_{t>0})(f)$, when $f\in L^p(\mathbb R^n,\mu_\nu )$, $1\leq p<\infty$, is in order. Since the function $t\in (0,\infty )\longrightarrow A_{\nu ,M,t}^k(f)(x)$, $x\in \mathbb R^n$, is continuous in $(0,\infty )$ (see section \ref{Svariation}), we can write
$$
V_\rho (\{A_{\nu,M,t}^k\}_{t>0})(f)(x)=\sup_{\substack{0<t_1<\cdots <t_\ell\\t_j\in \mathbb Q,\,j=1,\ldots, \ell,\,\ell \in \mathbb N}}\left(\sum_{j=1}^{\ell -1}|A_{\nu, M,t_j}^k(f)(x)-A_{\nu,M,t_{j+1}}^k(f)(x)|^\rho \right)^{1/\rho },\quad x\in \mathbb R^n.
$$
Then, since the set of the finite subsets of $\mathbb Q$ is countable, the measurability of $V_\rho (\{A_{\nu,M,t}^k\}_{t>0})(f)$ is justified.

In a similar way we can define the oscillation operator associated to a family of operators. Suppose that $\{t_j\}_{j\in \mathbb N}$ is a decreasing sequence in $(0,\infty )$ and $\{T_t\}_{t>0}$ a family of operators as before. We define
$$
O(\{T_t\}_{t>0},\{t_j\}_{j\in \mathbb N})(f)=\Big(\sum_{j\in \mathbb N}\;\;\sup_{t_{j+1}\leq \varepsilon <\varepsilon '\leq t_j}|T_\varepsilon(f)-T_{\varepsilon '}(f)|^2 \Big)^{1/2}.
$$
In particular, we consider the oscillation operator $O(\{A_{\nu,M,t}^k\}_{t>0},\{t_j\}_{j\in \mathbb N})$ associated with the family $\{A_{\nu ,M,t}^k\}_{t>0}$.

 Following \cite{LSj4} let $D^{(\ell)}$, $\ell \in \mathbb N$, a differentiation operator of the form $D^{(\ell)}=\sum_{|\alpha|=\ell}a_\alpha \partial ^\alpha$ where $\alpha =(\alpha_1,\ldots,\alpha_n)\in \mathbb N^n$, with $\alpha_i\not=0$ for some $i=1,\ldots,n$,  $|\alpha |=\alpha _1+\cdots +\alpha_n=\ell $ and $a_\alpha \in \mathbb C$. Here, for every $\alpha=(\alpha_1,\ldots,\alpha_n)\in \mathbb N^n$,  $\partial ^\alpha=\frac{\partial ^{\alpha_1+\cdots+\alpha_n}}{\partial x_1^{\alpha_1}\cdots\partial x_n^{\alpha _n}}$. It can be written $D^ {(\ell)}=\sum_{j=0 }^\ell\partial _\nu ^iD_{\ell-i}'$, where $\partial_\nu$ represents the differentiation operator along $\nu$ and $D_{\ell-i}'$ is a constant coefficient operator of order $\ell-i$ involving only differentiation in directions orthogonal to $\nu$. Let us denote by $\frak{q}$ the maximal order of differentiation along $\nu$, that is, $\frak q=\max\{i\in \{0,\ldots,\ell \}:D_{\ell-i}'\not=0\}$. Define, for every $\ell \in \mathbb{N}$, $1\leq \ell <2M$, the operator
$$
G_{\nu ,M}^{k,\ell}(f)=\left(\int_0^\infty \Big|t^{\ell /2+k}\partial_t^kD^{(\ell)}(I-t\Delta_\nu)^{-M}(f)\Big|^2\frac{dt}{t}\right)^{1/2}.
$$
In \cite[Theorem 1.2]{LSj4}, certain boundedness properties for other Littlewood-Paley functions involving the semigroups generated by $-\Delta_\nu$ were analyzed. The quantity $\frak{q}$ plays a relevant play.

Our main result is the following one.

\begin{thm}\label{Th1.1}
Let $\nu\in \mathbb R^n\setminus\{0\}$, $\rho >2$, $M>0$ and $k,\ell \in \mathbb N$, with $1\leq \ell <2M$. Assume that $\{t_j\}_{j\in \mathbb N}$ is a decreasing sequence in $(0,\infty )$. The operators $W_{\nu ,M,*}^k$, $g_{\nu, M}^k$ with $k\geq 1$, and $V_\rho (\{A_{\nu,M,t}^k\}_{t>0})$, $O(\{A_{\nu,M,t}^k\}_{t>0}, \{t_j\}_{j\in \mathbb N})$ when $M>n/2$, are bounded on $L^p(\mathbb R^n,\mu_\nu )$, $1<p<\infty$, and from $L^1(\mathbb R^n,\mu_\nu )$ into $L^{1,\infty}(\mathbb R^n,\mu_\nu )$. The operator $G_{\nu,M}^{k,\ell}$ is also bounded on $L^p(\mathbb R^n,\mu_\nu )$, $1<p<\infty$ and from $L^1(\mathbb R^n,\mu_\nu )$ into $L^{1,\infty}(\mathbb R^n,\mu_\nu )$ provided that $\mathfrak q\leq 2$. In addition, when $\mathfrak q>1$, there exists $C>0$ such that, for every $f\in L(1+\ln^+L)^{\mathfrak q/2-3/4}$,
$$
\mu_\nu (\{x\in \mathbb R^n:G_{\nu,M}^{k,\ell}(f)(x)>\lambda\})\leq C\int_{\mathbb R^n}\frac{|f(x)|}{\lambda}\Big(1+\ln^+\frac{|f(x)|}{\lambda}\Big)^{\mathfrak q/2-3/4}d\mu_\nu (x).
$$
\end{thm}

We have organized this paper by describing firstly the elements and  procedures used in our proofs (section 2) and then by developing the complete proofs in the subsequent sections.  

Throughout this paper $C$ and $c$ denote positive constants that can change in each occurrence.
\section{Preliminaries}
\subsection{Local and global parts of an operator}\label{SS2}
For every $\eta>0$ we define
$$
L_\eta=\{(x,y)\in \mathbb R^n\times \mathbb R^n: |x-y|<\eta\},
$$
and $G_\eta = (\mathbb R^n\times \mathbb R^n)\setminus L_\eta$. The letters $L$ and $G$ mean local and global, respectively.

We begin stating a property similar to the one showed in $\!\!$\cite{GCMST1} for the Gaussian setting which can be proved in a straightforward way.

\begin{prop}\label{Prop2.1}
There exists a sequence $\{x_j\}_{j\in \mathbb N}\subset \mathbb R^n$ such that

(i) $\mathbb R^n=\bigcup_{j\in \mathbb N}B(x_j,\frac{1}{10})$;

(ii) The family $\{B(x_j,\frac{1}{40})\}_{j\in \mathbb N}$ are pairwise disjoint;

(iii) For every $A>0$, $\sup_{x\in \mathbb R^n}\sum_{j\in \mathbb N}\mathcal X_{B(x_j,\frac{A}{10})}(x)<\infty$;

(iv) $L_{3/10}\subset \bigcup_{j\in \mathbb N}\big[B(x_j,\frac{1}{10})\times B(x_j,\frac{4}{10})\big]\subset L_{1/2}$;

(v) For every $\eta >0$ there exists $A>0$ such that $B(x_j,\frac{1}{10})\times \big[\mathbb R^n\setminus B(x_j,\frac{A}{10})\big]\subset G_\eta$;

(vi) If $A>0$ there exists $C>1$ such that, for every $j\in \mathbb N$,
$$
\frac{1}{C}e^{2\langle \nu,x_j\rangle}\int_E|f(x)|dx\leq \int_E|f(x)|d\mu_\nu(x)\leq Ce^{2\langle \nu, x_j\rangle }\int_E|f(x)|dx,\quad f\in L^1_{\rm loc}(\mathbb R^n),
$$
for every measurable set $E\subset B(x_j,\frac{A}{10})$. 
\end{prop}

Let $\mathbb B_1$ and $\mathbb B_2$ be Banach spaces. Suppose that $T$ is a bounded linear operator from $L_{\mathbb B_1}^q(\mathbb R^n,\mu_\nu)$ into $L_{\mathbb B_2}^q(\mathbb R^n,\mu_\nu)$ or from  $L_{\mathbb B_1}^q(\mathbb R^n,\mu_\nu)$ into $L_{\mathbb B_2}^{q,\infty}(\mathbb R^n,\mu_\nu)$, for some $1\leq q<\infty$.

We define a function $\psi \in C^\infty (\mathbb R)$ satisfying that

(i) $0\leq \psi \leq 1$,

(ii) $\psi (x)=1$, $|x|<1$, and $\psi (x)=0$, $|x|\geq 2$,\vspace*{0.1cm}

(iii) $\displaystyle |\psi'(x)|\leq \frac{C}{|x|}$, $x\in \mathbb R\setminus\{0\}$, for some $C>0$,\vspace*{0.1cm}

\noindent and consider the function $\varphi (x,y)=\psi (|x-y|)$, $x,y\in \mathbb R^n$. It follows that

    (a) $\varphi (x,y)\in [0,1]$,

    (b) $\varphi (x,y)=1$, $|x-y|\leq 1$, and  $\varphi (x,y)=0$, $|x-y|\geq 2$,

    (c) $|\nabla _x\varphi (x,y)|+|\nabla _y\varphi (x,y)|\leq \frac{C}{|x-y|}$, $x\not=y$.\vspace*{0.1cm}

\noindent We decompose the operator $T$ as $T(f)=T_{\rm loc}(f)+T_{\rm glob}(f)$, where $T_{\rm loc}(f)(x)=T(f(\cdot)\varphi (x,\cdot ))(x)$, $x\in \mathbb R^n$.

\begin{prop}\label{Prop2.2}
Let $\mathbb B_1$ and $\mathbb B_2$ be Banach spaces and $1\leq q<\infty$. Suppose that $T$ is a bounded linear operator from  $L_{\mathbb B_1}^q(\mathbb R^n,\mu_\nu)$ into $L_{\mathbb B_2}^q(\mathbb R^n,\mu_\nu)$ (respectively, $L_{\mathbb B_2}^{q,\infty}(\mathbb R^n,\mu_\nu)$). Assume that there exists a function $K:(\mathbb R^n\times \mathbb R^n)\setminus \mathfrak D\longrightarrow \mathcal L(\mathbb B_1,\mathbb B_2)$, where $\mathfrak D=\{(x,x),x\in \mathbb R^n\}$ and $\mathcal L(\mathbb B_1,\mathbb B_2)$ denotes the space of linear bounded operators from $\mathbb B_1$ into $\mathbb B_2$, such that, for every $f\in L_c^\infty (\mathbb R^n,\mathbb B_1)$, the space of essentially bounded $\mathbb B_1$-valued functions with compact support,
$$
T(f)(x)=\int_{\mathbb R^n}K(x,y)f(y)dy,\quad x\not\in \supp f,
$$
and
$$
\|K(x,y)\|_{\mathcal L(\mathbb B_1,\mathbb B_2)}\leq \frac{C}{|x-y|^n},\quad x,y\in \mathbb R^n,\,0<|x-y|<2.
$$
Then, $T_{\rm loc}$ is bounded from  $L_{\mathbb B_1}^q(\mathbb R^n,\mu_\nu)$ into $L_{\mathbb B_2}^q(\mathbb R^n,\mu_\nu)$ (respectively, $L_{\mathbb B_2}^{q,\infty}(\mathbb R^n,\mu_\nu)$) and from $L_{\mathbb B_1}^q(\mathbb R^n,dx)$ into $L_{\mathbb B_2}^q(\mathbb R^n,dx)$ (respectively, $L_{\mathbb B_2}^{q,\infty}(\mathbb R^n,dx)$).
\end{prop}

\begin{proof}
Assume that $T$ is a bounded linear operator from $L_{\mathbb B_1}^q(\mathbb R^n,\mu_\nu)$ into $L_{\mathbb B_2}^q(\mathbb R^n,\mu_\nu)$.
    Denote, for every $j\in \mathbb N$, $B_j=B(x_j,\frac{1}{10})$, where $\{x_j\}_{j\in \mathbb N}$ is the sequence in Proposition \ref{Prop2.1}. As in \cite[Lemma 3.2.6]{Sa} we define, for every $f\in C_c^\infty (\mathbb R^n,\mathbb B_1)$,
    $$
    \mathbb T (f)=\sum_{j\in \mathbb N}\mathcal X_{B_j}T(\mathcal X_{4B_j}f).
    $$
     Let us see that $\mathbb T$ is bounded from $L_{\mathbb B_1}^q(\mathbb R^n,\mu_\nu)$ into $L_{\mathbb B_2}^q(\mathbb R^n,\mu_\nu)$ and from  $L_{\mathbb B_1}^q(\mathbb R^n,dx)$ into $L_{\mathbb B_2}^q(\mathbb R^n,dx)$. Let $f\in C_c^\infty (\mathbb R^n,\mathbb B_1)$. By taking into account the properties of the sequence of balls $\{B_j\}_{j\in \mathbb N}$ given in Proposition \ref{Prop2.1} it follows that
\begin{align*}
\big\|\mathbb T(f)\big\|_{L^q_{\mathbb B_2}(\mathbb R^n,\mu_\nu)}^q&=\int_{\mathbb R^n}\big\|\sum_{j\in \mathbb N}\mathcal X_{B_j}(x)T(\mathcal X_{4B_j}f)(x)\big\|_{\mathbb B_2}^qd\mu _\nu (x)\leq C\sum_{j\in \mathbb N}\int_{\mathbb R^n}\big\|T(\mathcal X_{4B_j}f)(x)\big\|_{\mathbb B_2}^qd\mu _\nu (x)\\
&\leq C\sum_{j\in \mathbb N}\int_{4B_j}\big\|f(x)\big\|_{\mathbb B_1}^qd\mu _\nu (x)\leq C\big\|f\big\|_{L^q_{\mathbb B_1}(\mathbb R^n,\mu_\nu )}^q.
\end{align*}
On the other hand, Proposition \ref{Prop2.1} (vi) leads to
\begin{align*}
\big\|\mathbb T(f)\big\|_{L^q_{\mathbb B_2}(\mathbb R^n,dx)}^q&\leq C\sum_{j\in \mathbb N}\int_{B_j}\big\|T(\mathcal X_{4B_j}f)(x)\big\|_{\mathbb B_2}^qdx\leq C\sum_{j\in \mathbb N}\int_{B_j}e^{-2\langle \nu,x_j\rangle}\big\|T(\mathcal X_{4B_j}f)(x)\big\|_{\mathbb B_2}^qd\mu _\nu (x)\\
&\hspace{-1cm}\leq C\sum_{j\in \mathbb N}e^{-2\langle \nu,x_j\rangle}\int_{4B_j}\big\|f(x)\big\|_{\mathbb B_1}^qd\mu_\nu(x)\leq C\sum_{j\in \mathbb N}\int_{4B_j}\big\|f(x)\big\|_{\mathbb B_1}^qdx\leq C\|f\|_{L^q_{\mathbb B_1}(\mathbb R^n,dx)}^q.
\end{align*}

To show the boundedness properties for $T_{\rm loc}$, as in \cite[p. 89]{Sa} we write
$$
T_{\rm loc}(f)(x)=T(\mathcal X_{4B_j}f)(x)+\int_{\mathbb R^n}(\varphi (x,y)-\mathcal X_{4B_j}(y))K(x,y)f(y)dy,\quad x\in \mathbb R^n\mbox{ and }j\in \mathbb N,
$$
from which it is deduced that
$$
\|T_{\rm loc}(f)(x)\|_{\mathbb B_2}\leq \|\mathbb T(f)(x)\|_{\mathbb B_2}+\tau (\|f\|_{\mathbb B_1})(x),\quad x\in \mathbb R^n,
$$
where
$$
\tau (\|f\|_{\mathbb B_1})(x)=\sum_{j\in \mathbb N}\mathcal X_{B_j}(x)\int_{\mathbb R^n}|\varphi (x,y)-\mathcal X_{4B_j}(y)|\|K(x,y)\|_{\mathcal L(\mathbb B_1,\mathbb B_2)}\|f(y)\|_{\mathbb B_1}dy,\quad x\in \mathbb R^n.
$$
Let us show that for every $1\leq p\leq \infty$, 
\begin{equation}\label{interpdx}
\|\tau (\|f\|_{\mathbb B_1})\|_{L^p(\mathbb R^n,dx)}\leq C\|f\|_{L^p_{\mathbb B_1}(\mathbb R^n,dx)},
\end{equation}
and 
\begin{equation}\label{interpdmu}
\|\tau (\|f\|_{\mathbb B_1})\|_{L^p(\mathbb R^n,\mu_\nu)}\leq C\|f\|_{L^p_{\mathbb B_1}(\mathbb R^n,\mu_\nu)}.
\end{equation}
Thus, we can conclude that $T_{\rm loc}$ is bounded from $L^q_{\mathbb B_1}(\mathbb R^n,\mu_\nu)$ into $L^q_{\mathbb B_2}(\mathbb R^n,\mu_\nu)$ and also, from $L^q_{\mathbb B_1}(\mathbb R^n,dx)$ into $L^q_{\mathbb B_2}(\mathbb R^n,dx)$.

Observe that if $x\in B_j$ and $\varphi (x,y)-\mathcal X_{4B_j}(y)\not=0$, then $(x,y)\in L_2\setminus L_{3/10}$ and, consequently,
$$
\tau (\|f\|_{\mathbb B_1})(x)\leq C\sum_{j\in \mathbb N}\mathcal X_{B_j}(x)\int_{|x-y|<2}\|f(y)\|_{\mathbb B_1}dy\leq C\int_{|x-y|<2}\|f(y)\|_{\mathbb B_1}dy,\quad x\in \mathbb R^n.
$$
Thus,
$$
\|\tau (\|f\|_{\mathbb B_1})\|_{L^1(\mathbb R^n,dx)}\leq C\int_{\mathbb R^n}\|f(y)\|_{\mathbb B_1}\int_{|x-y|<2}dxdy\leq C\|f\|_{L^1_{\mathbb B_1}(\mathbb R^n,dx)}, 
$$
and 
$$
\|\tau (\|f\|_{\mathbb B_1})\|_{L^\infty(\mathbb R^n,dx)}\leq C\|f\|_{L^\infty_{\mathbb B_1}(\mathbb R^n,dx)}.
$$
By using interpolation \eqref{interpdx} is obtained.

On the other hand, we can write 
\begin{align*}
\int_{\mathbb R^n}\tau (\|f\|_{\mathbb B_1})(x)d\mu_\nu(x)&\leq C\int_{\mathbb R^n}\|f(y)\|_{\mathbb B_1}\mu_\nu (B(y,2))dy\leq C\int_{\mathbb R^n}\|f(y)\|_{\mathbb B_1}\int_{|x-y|<2}e^{2x_1}dxdy\\
&\leq C\int_{\mathbb R^n}\|f(y)\|_{\mathbb B_1}e^{2y_1}\int_{|x-y|<2}e^{2(x_1-y_1)}dxdy\leq C\int_{\mathbb R^n}\|f(y)\|_{\mathbb B_1}d\mu_\nu (y),
\end{align*}
and also $\tau (\|f\|_{\mathbb B_1})\|_{L^\infty (\mathbb R^n,\mu _\nu )}\leq \|f\|_{L^\infty_{\mathbb B_1}(\mathbb R^n,\mu_\nu )}$ (Note that $L^\infty (\mathbb R^n,\mu_\nu)=L^\infty (\mathbb R^n,dx)$). Again interpolation leads to \eqref{interpdmu}.

If $T$ is a bounded linear operator from $L^q_{\mathbb B_1}(\mathbb R^n,\mu_\nu )$ into $L^{q,\infty}_{\mathbb B_2}(\mathbb R^n,\mu_\nu )$ the property can be proved in a similar way.
\end{proof}

\begin{rem}
    The properties in Proposition \ref{Prop2.2} also hold when we assume that $T$ is a bounded linear operator from $L_{\mathbb B_1}^q(\mathbb R^n,dx)$ into $L^q_{\mathbb B_2}(\mathbb R^n,dx)$ (respectively, $L^{q,\infty }_{\mathbb B_2}(\mathbb R^n,dx)$).
\end{rem}

\begin{prop}\label{Prop2.3}
Let $\mathbb B_1$ and $\mathbb B_2$ be Banach spaces and $1\leq q<\infty$. Suppose that $T$ is a linear operator defined in $L^q_{\mathbb B_1}(\mathbb R^n,dx)$ (respectively, $L^q_{\mathbb B_1}(\mathbb R^n,\mu_\nu)$) and taking values in $L^{q,\infty}_{\mathbb B_2}(\mathbb R^n,dx)$ (respectively, $L^{q,\infty}_{\mathbb B_2}(\mathbb R^n,\mu_\nu)$). Then the following properties are equivalent.

(a) $T_{\rm loc}$ is bounded from $L^q_{\mathbb B_1}(\mathbb R^n,dx)$ into $L^q_{\mathbb B_2}(\mathbb R^n,dx)$ (respectively, $L^{q,\infty}_{\mathbb B_2}(\mathbb R^n,dx)$).

(b) $T_{\rm loc}$ is bounded from $L^q_{\mathbb B_1}(\mathbb R^n,\mu_\nu)$ into $L^q_{\mathbb B_2}(\mathbb R^n,\mu_\nu )$ (respectively, $L^{q,\infty}_{\mathbb B_2}(\mathbb R^n,\mu_\nu)$).
\end{prop}
\begin{proof}
 Consider the family of balls $\{B_j=B(x_j,\frac{1}{10})\}_{j\in \mathbb N}$ given in Proposition \ref{Prop2.1}. Assume that $T_{\rm loc}$ is a bounded operator from $L^q_{\mathbb B_1}(\mathbb R^n,dx)$ into $L^q_{\mathbb B_2}(\mathbb R^n,dx)$. According to properties in Proposition \ref{Prop2.1} we can write
$$
\int_{\mathbb R^n}\|T_{\rm loc}(f)(x)\|_{\mathbb B_2}^qd\mu_\nu (x)\leq C\sum_{j\in \mathbb N}\int_{B_j}\|T_{\rm loc}(f)(x)\|_{\mathbb B _2}^qd\mu_\nu (x)\leq C \sum_{j\in \mathbb N}e^{\langle \nu,x_j\rangle}\int_{B_j}\|T_{\rm loc}(f)(x)\|_{\mathbb B _2}^qdx.
$$
We now observe that for every $j\in \mathbb N$, $T_{\rm loc}(f)(x)= T_{\rm loc}(f\mathcal X_{21B_j})(x)$, $x\in B_j$. Thus, the boundedness property for $T_{\rm loc}$ and Proposition \ref{Prop2.1} (iii) and  (vi),  lead to
\begin{align*}
\int_{\mathbb R^n}\|T_{\rm loc}(f)(x)\|_{\mathbb B_2}^qd\mu_\nu (x)&\leq C \sum_{j\in \mathbb N}e^{\langle \nu,x_j\rangle}\int_{21B_j}\|f(x)\|_{\mathbb B _1}^qdx\leq C\int_{\mathbb R^n}\|f(x)\|_{\mathbb B _1}^qd\mu_\nu(x).
\end{align*}

Similar reasoning can be followed to establish that {\it (b)} implies {\it (a)}, and also that the properties are equivalent when the operators are of weak type $(q,q)$.
\end{proof}

\subsection{About the proof of Theorem \ref{Th1.1}}

Let $k\in \mathbb N$ and $M>0$. By arguing as in \cite[\S 1.1]{LSjW}, it can be seen that it is enough to establish our results when $\nu=(1,0,\ldots,0)$. Thus, from now on we assume that $\nu=(1,0,\ldots,0)$. For every $t>0$, $W_t^\nu$ is an integral operator given, for every $f\in L^p(\mathbb R^n, \mu_\nu)$, $1\leq p < \infty$, by
$$
W_t^\nu(f)(x) = \int_{\mathbb R^n} W_t^\nu(x,y) f(y)d\mu_\nu(y),\;\; x\in \mathbb R^n,
$$
where
$$
W_t^\nu(x,y) = \frac{1}{(4\pi t)^{\frac{n}{2}}}e^{-x_1-y_1}e^{-t}e^{-\frac{|x-y|^2}{4t}},\;\; x=(x_1,x'),\,y=(y_1,y') \in \mathbb R\times \mathbb R^{n-1}.
$$

We also observe that the operators under consideration are almost linear (\!\cite[p. 481, Def. 1.20]{GCRdF}). Indeed, we can write $W_{\nu,M,*}^k(f)=\|A_{\nu ,M,\cdot}^k(f)\|_{L^\infty ((0,\infty ),dt)}$, $g_{\nu,M}^k(f)=\|A_{\nu ,M,\cdot}^k(f)\|_{L^2((0,\infty ),\frac{dt}{t})}$, $V_\rho (\{A_{\nu ,M,t}^k\}_{t>0})(f)=\|A_{\nu ,M,\cdot}^k(f)\|_{E_\rho}$, 
$O(\{A_{\nu ,M,t}^k\}_{t>0}, \{t_j\}_{j\in \mathbb N})(f)=\|A_{\nu ,M,\cdot}^k(f)\|_{O(\{t_j\}_{j\in \mathbb N})}$ and
$$
G_{\nu,M}^{k,\ell}(f)=\|t^{\ell/2+k}\partial_t^kD^{(\ell)}(I-t\Delta_\nu)^{-M}(f)\|_{L^2((0,\infty ),\frac{dt}{t})}.
$$

Let $f\in C_c^\infty (\mathbb R^n)$, the space of smooth functions with compact support in $\mathbb R^n$. We have that
\begin{equation}\label{If}
(I-t\Delta_\nu )^{-M}f(x)=\frac{t^{-M}}{\Gamma (M)}\int_0^\infty e^{-u/t}W_u^\nu (f)(x)u^{M-1}du,\quad x\in \mathbb R^n.
\end{equation}
We can write
$$
t^k\partial_t^k(I-t\Delta_\nu )^{-M}f(x)=\frac{1}{\Gamma (M)}\int_0^\infty t^k\partial _t^k(t^{-M}e^{-u/t})W_u^\nu (f)(x)u^{M-1}du,\quad x\in  \mathbb R^n.
$$
Then, the $L^p$-boundedness properties of the maximal operator will be deduced from the corresponding properties of the maximal operator defined by the semigroup $\{W_t^\nu \}_{t>0}$ established in \cite[Theorem, p. 73]{StLP} and \cite[Theorem 2]{LSjW}. We also obtain adequate bounds for the operator $G_{\nu,M}^{k,\ell}$ which allow us to proceed as in \cite[Theorem 1.2]{LSj4} to get our result.

In order to prove Theorem \ref{Th1.1} for the remainder operators we proceed by using a localization argument. We decompose, for every $t>0$, the operator $A_{\nu,M,t}^k$ as follows
$$
A_{\nu ,M,t}^k(f)=A_{\nu,M,t,{\rm loc}}^k(f)+A_{\nu,M,t,{\rm glob}}^k(f),
$$
where the local and the global operators are defined as in section \ref{SS2}. The local operator is studied by using vector-valued Calder\'on-Zygmund theory (\!\! \cite{RubRT}) by considering in each case the suitable Banach space. On the other hand, the global operator $\|A_{\nu,M,t,{\rm glob}}^k(f)\|_X$, where $X=L^2((0,\infty ),\frac{dt}{t})$, $X=E_\rho$ and $X=O(\{t_j\}_{j\in \mathbb N})$ are controlled by a positive operator having the $L^p$-boundedness properties that we need. 

\section{Proof of Theorem \ref{Th1.1} for maximal operators}
Since $\{W_t^\nu\}_{t>0}$ is a symmetric diffusion semigroup in the sense of \cite[p. 65]{StLP}, the operator $W_t^\nu $, $t>0$, is contractively regular in $L^p(\mathbb R^n,\mu _\nu)$, $1\leq p\leq \infty$ (see \cite[\S 2]{LeMX1}). Then, according to \cite[Corollary 4.2]{LeMX1}, for every $m\in \mathbb N$, the maximal operator $W_{*,m}^\nu$ defined by
$$
W_{*,m}^\nu(f)=\sup_{t>0}|t^m\partial _t^mW_t^\nu (f)|
$$
is bounded on $L^p(\mathbb R^n,\mu_\nu)$, for every $1<p<\infty$. Moreover, by \!\! \cite [Theorem 2]{LSjW}, $W_{*,0}^\nu$ is bounded from $L^1(\mathbb R^n,\mu_\nu)$ into $L^{1,\infty}(\mathbb R^n,\mu_\nu)$. 

Let $k\in \mathbb N$ and $M>0$. Let us show that, for certain $C>0$, $W_{\nu ,M,*}^k(f)\leq CW_{*,0}^\nu (f)$, for every $f\in C_c^\infty (\mathbb R^n)$. Thus, the $L^p$-boundedness properties for $W_{\nu, M,*}^k$ are obtained from the corresponding ones for $W_{*,0}^\nu$.

First, we observe that, for each $t>0$ and $x\in \mathbb R^n$,
\begin{align*}
    (I-t\Delta_\nu )^{-M}f(x)&=t^{-M}\Big(\frac{1}{t}I-\Delta _\nu \Big)^{-M}(f)(x)=\frac{t^{-M}}{\Gamma (M)}\int_0^\infty e^{-(\frac{1}{t}I-\Delta _\nu )u}f(x)u^{M-1}du\\
    &=\frac{t^{-M}}{\Gamma (M)}\int_0^\infty e^{-u/t}W_u^\nu (f)(x)u^{M-1}du,
\end{align*}
and
\begin{equation}\label{A_nuMt^k}
A_{\nu,M,t}^k(f)(x)=t^k\partial _t^k(I-t\Delta _\nu)^{-M}(f)(x)=\frac{t^k}{\Gamma (M)}\int_0^\infty \partial _t^k(t^{-M}e^{-u/t})W_u^\nu (f)(x)u^{M-1}du.
\end{equation}

\noindent Then, $|A_{\nu,M,t}^k(f)(x)|\leq CH(t)W_{*,0}^\nu(f)(x)$, $t>0$, $x\in \mathbb R^n$, where
$$
H(t)=\int_0^\infty |t^k\partial _t^k(t^{-M}e^{-u/t})|u^{M-1}du,\quad t>0.
$$

Denote by $\mathcal N_k=\{(m_1,...,m_k)\in \mathbb N^k: m_1+2m_2+\cdots +km_k=k\}$. By using Fa\`a di Bruno formula we get
\begin{align}\label{1.1}
    t^k\partial _t^k(t^{-M}e^{-u/t})&=t^k\sum_{(m_1,\ldots,m_k)\in \mathcal N_k}\frac{k!}{m_1!1!^{m_1}\cdots m_k!k!^{m_k}}\prod_{\ell=1}^k\left(\frac{(-1)^\ell \ell!}{t^{\ell+1}}\right)^{m_\ell}\nonumber\\
    &\quad \times \sum_{j=0}^{m_1+\cdots+m_k}\Big(\substack{m_1+\cdots +m_k\\j}\Big)M(M-1)\cdots (M-j+1)\Big(\frac{1}{t}\Big)^{M-j}(-u)^{m_1+\cdots+m_k-j}e^{-u/t}\nonumber\\
    &=\sum_{(m_1,\ldots,m_k)\in \mathcal N_k}\!\!\!\sum_{j=0}^{m_1+\cdots+m_k}a_{m_1,\ldots,m_k,j}t^{-M}\Big(\frac{u}{t}\Big)^{m_1+\cdots +m_k-j}e^{-u/t}\nonumber\\
    &=:\sum_{(m_1,\ldots,m_k)\in \mathcal N_k}\!\!\!\sum_{j=0}^{m_1+\cdots+m_k}a_{m_1,\ldots,m_k,j}H_{M,m_1,\ldots,m_k,j}(u,t),\quad u,t>0.
\end{align}
Here $a_{m_1,\ldots,m_k,j}\in \mathbb R$, for every $(m_1,\ldots, m_k)\in \mathcal N_k$ and $j\in \mathbb N$, $0\leq j\leq m_1+\cdots+m_k$. It follows that
\begin{equation}\label{1.2}
H(t)\leq C\sum_{(m_1,\ldots,m_k)\in \mathcal N_k}\!\!\!\sum_{j=0}^{m_1+\cdots+m_k}\int_0^\infty e^{-v}v^{m_1+\cdots +m_k-j+M-1}dv=C,\quad t>0,
\end{equation}
from which we deduce that $W_{\nu,M,*}^k(f)(x)\leq CW_{*,0}^\nu (f)$.

\section{Proof of Theorem \ref{Th1.1} for Littlewood-Paley functions, $g_{\nu,M}^k$}

Since $\{W_t^\nu\}_{t>0}$ is a symmetric diffusion semigroup in the sense of \cite{StLP}, $\{W_t^\nu\}_{t>0}$ can be extended to a sector $\Sigma_{\alpha _p}=\{z\in \mathbb C\setminus\{0\}: |{\rm Arg }\,z|<\alpha _p\}$, for certain $0<\alpha _p\leq 1$, and being $\{W_z^\nu\}_{z\in \Sigma_{\alpha_p} }$ a bounded analytic semigroup in $L^p(\mathbb R^n,\mu_\nu)$, $1<p<\infty$ (\!\! \cite[Theorem 1, p. 67]{StLP}). Then, $-\Delta_\nu$ is a sectorial operator in $L^p(\mathbb R^n,\mu_\nu)$, $1<p<\infty$, ($\!\!$\cite[Theorem 5.2, \S 2.5]{Pa}). In addition, the range of $\Delta_\nu$ contains the set $\prod_{j=1}^nC_c^\infty (\mathbb R)$, where $C_c^\infty (\mathbb R)$ denotes the smooth and compactly supported functions in $\mathbb R$. Then, the range of $\Delta _\nu$ is dense in $L^p(\mathbb R^n,\mu_\nu)$, for every $1<p<\infty$.

Let $k\in \mathbb N$, $k\geq 1$ and $M>0$. We consider the function
$$
F_{k,M}(z)=(-1)^k\frac{\Gamma (M+k)}{\Gamma (M)}\frac{z^k}{(1+z)^{M+k}},\quad z\in \mathbb C\setminus \{-1\}.
$$
We have that
$$
|F_{k,M}(z)|\leq C\frac{|z|^k}{(1+|z|)^k},\quad z\in \Sigma_\theta,\,0<\theta <\alpha _p.
$$
Then, according to \cite[Theorem 6.6]{CDMY} (see also, \cite[Proposition 2.1]{LeMX1}), for every $1<p<\infty$, there exists $C>0$ such that
\begin{equation}\label{1.0}
\Big\|\Big(\int_0^\infty |F_{k,M}(-t\Delta _\nu )f|^2\frac{dt}{t}\Big)^{1/2}\Big\|_{L^p(\mathbb R^n,\mu_\nu)}\leq C\|f\|_{L^p(\mathbb R^n,\mu_\nu)},\quad f\in L^p(\mathbb R^n,\mu_\nu).    
\end{equation}
Since $A_{\nu,M,t}^k=F_{k,M}(-t\Delta _\nu)$, the Littlewood-Paley function $g_{k,M}^\nu$ is bounded on $L^p(\mathbb R^n,\mu_\nu)$, $1<p<\infty$.

Next we are going to show that $g_{\nu,M}^k$ is bounded from $L^1(\mathbb R^n,\mu_\nu)$ into $L^{1,\infty}(\mathbb R^n,\mu_\nu)$. In order to establish this property we see $g_{\nu,M}^k$ as a $L^2((0,\infty ),\frac{dt}{t})$-valued operator.

Let $N\in \mathbb N$, $N\geq 1$, and $f\in C_c^\infty (\mathbb R^n)$, the space of smooth functions with compact support in $\mathbb R^n$. Consider the function $G_{N,f}:\mathbb R^n\longrightarrow \mathbb B_N:=L^2((1/N,N),\frac{dt}{t})$ given by 
$$
[G_{N,f}(x)](t)=A_{\nu,M,t}^k (f)(x), \quad x\in \mathbb R^n, \,t\in (1/N,N). 
$$
The function $G_{N,f}$ is strongly measurable. Indeed, since $\mathbb B_N$ is a separable Banach space, according to Pettis theorem, it is sufficient to prove that $G_{N,f}$ is weakly measurable. Let $h\in \mathbb B_N$ and let $G_{N,f,h}$ the function defined on $\mathbb R^n$ by 
$$
G_{N,f,h}(x)=\int_{1/N}^NA_{\nu,M,t}^k(f)(x)h(t)\frac{dt}{t},\quad x\in \mathbb R^n.
$$
This function is continuous in $\mathbb R^n$. Indeed, let $x_0\in \mathbb R^n$. 
By taking into account \cite[(1.8)]{LSj4} and that $f\in C_c^\infty (\mathbb R^n)$ we have that
$$
|W_u^\nu (f)(x)-W_u^\nu (f)(x_0)|\leq Ce^{-u}|x-x_0|\int_{\mathbb R^n}|f(y)|d\mu _\nu (y)\leq Ce^{-u}|x-x_0|, \quad |x-x_0|<1\mbox{ and }u>0.
$$
Then, when $|x-x_0|<1$, according to \eqref{A_nuMt^k}, \eqref{1.1} and \eqref{1.2} we conclude that
\begin{align*}
    |G_{N,f,h}(x)-G_{N,f,h}(x_0)|\leq C|x-x_0|\int_{1/N}^N|h(t)|\int_0^\infty |t^k\partial _t^k(t^{-M}e^{-u/t})|u^{M-1}du\frac{dt}{t}\leq C|x-x_0|,
    \end{align*} 
and the continuity of $G_{N,f,h}$ on $x_0$ is proved. 

We now define the operator $G_N$ by $G_N(f)=G_{N,f}$, $f\in C_c^\infty (\mathbb R^n)$, and decompose $G_N$ as follows
$G_N(f)=G_{N,{\rm loc}}(f)+G_{N,{\rm glob}}(f)$, where the local and global parts are defined as in section \ref{SS2}.

We study $G_{N,{\rm loc}}$. Consider the function
\begin{equation}\label{K}
K_{\nu,M,t}^k (x,y)=\frac{e^{2y_1}}{\Gamma (M)}\varphi (x,y)t^k\partial _t^k\Big[t^{-M}\int_0^\infty e^{-u/t}W_{u}^\nu (x,y)u^{M-1}du\Big],\quad x,y\in \mathbb R^n\mbox{ and }t>0.
\end{equation}
We have that
\begin{align}\label{A1}
K_{\nu,M,t}^k (x,y)&=\frac{e^{2y_1}}{\Gamma (M)}\varphi (x,y)t^k\partial _t^k\Big[\int_0^\infty e^{-v}W_{tv}^\nu (x,y)v^{M-1}dv\Big]\nonumber\\
&=\frac{e^{2y_1}}{\Gamma (M)}\varphi (x,y)t^k\partial _t^{k-1}\Big[\int_0^\infty e^{-v}\partial _uW_u^\nu (x,y)_{|u=tv}v^Mdv\Big]\nonumber\\
&=\frac{e^{2y_1}}{\Gamma (M)}\varphi (x,y)t^k\partial _t^{k-1}\Big[t^{-M-1}\int_0^\infty e^{-u/t}\partial _uW_u^\nu (x,y)u^Mdu\Big]\nonumber\\
&=\frac{e^{2y_1}}{\Gamma (M)}\varphi (x,y)\int_0^\infty t^k\partial _t^{k-1}[t^{-M-1}e^{-u/t}]\partial _uW_u^\nu (x,y)u^Mdu,\quad x,y\in \mathbb R^n\mbox{ and }t>0.
\end{align}
According to \eqref{1.1} we get, for every $x,y\in \mathbb R^n$ and $t>0$
\begin{align}\label{1.3}
K_{\nu,M,t}^k (x,y)&=e^{2y_1}\varphi (x,y)\sum_{(m_1,\ldots,m_{k-1})\in \mathcal N_{k-1}}\sum_{j=0}^{m_1+\cdots+m_{k-1}}a_{m_1,\ldots,m_{k-1},j}\nonumber\\
&\quad \times \int_0^\infty H_{M+1,m_1,\ldots,m_{k-1},j}(u,t)\partial_uW_u^\nu (x,y)u^Mdu\nonumber\\
&=\sum_{(m_1,\ldots,m_{k-1})\in \mathcal N_{k-1}}\sum_{j=0}^{m_1+\cdots+m_{k-1}}a_{m_1,\ldots,m_{k-1},j}\varphi(x,y)\mathcal S_{m_1+m_2+\cdots m_{k-1}+M-j}(x,y,t),
\end{align}
where, for every $r>0$,
\begin{equation}\label{Sl}
\mathcal S_r (x,y,t)=e^{2y_1}\int_0^\infty \Big(\frac{u}{t}\Big)^r e^{-u/t}\partial _uW_u^\nu (x,y)du,\quad x,y\in \mathbb R^n\mbox{ and }t>0.
\end{equation}

Let $r>0$. Let us see that for each $x,y\in \mathbb R^n$, $x\not=y$, the following two properties hold:

(a) $\displaystyle \|\varphi(x,y)\mathcal S_r (x,y,\cdot )\|_{L^2((0,\infty ),\frac{dt}{t})}\leq \frac{C}{|x-y|^n}$. 

(b) $\displaystyle \| |\nabla_x[\varphi(x,y)\mathcal S_r (x,y,\cdot )]|\|_{L^2((0,\infty ),\frac{dt}{t})}+\| |\nabla_y[\varphi(x,y)\mathcal S_r (x,y,\cdot )]|\|_{L^2((0,\infty ),\frac{dt}{t})}\leq \frac{C}{|x-y|^{n+1}}$. 
\vspace{0.2cm}

By using Minkowski's inequality we obtain, for every $x,y\in \mathbb R^n$, $x\not=y$,
\begin{align*}
\|\varphi(x,y)\mathcal S_r(x,y,\cdot )\|_{L^2((0,\infty ),\frac{dt}{t})}&\leq \|\mathcal S_r(x,y,\cdot )\|_{L^2((0,\infty ),\frac{dt}{t})}\\
&\leq e^{2y_1}\int_0^\infty u^r |\partial_uW_u^\nu (x,y)|\Big(\int_0^\infty \Big(\frac{e^{-u/t}}{t^r}\Big)^2\frac{dt}{t}\Big)^{1/2}du\\
&\leq C\int_0^\infty |\partial _u[e^{2y_1}W_u^\nu (x,y)]|du,
\end{align*}
and in a similar way and considering the properties of the function $\varphi$, 
\begin{align*}
\|\partial _{x_j}[\varphi(x,y)\mathcal S_r (x,y,\cdot )]\|_{L^2((0,\infty ),\frac{dt}{t})}+\|\partial_{y_j}[\varphi(x,y)\mathcal S_r (x,y,\cdot )]\|_{L^2((0,\infty ),\frac{dt}{t})}&\\
&\hspace{-10cm}\leq \big(|\partial _{x_j}\varphi(x,y)|+|\partial_{y_j}\varphi(x,y)|\big)\|\mathcal S_r (x,y,\cdot )]\|_{L^2((0,\infty ),\frac{dt}{t})}\\
&\hspace{-10cm} \quad +|\varphi(x,y)|\big(\|\partial _{x_j}\mathcal S_r (x,y,\cdot )\|_{L^2((0,\infty ),\frac{dt}{t})}+\|\partial_{y_j}\mathcal S_r (x,y,\cdot )\|_{L^2((0,\infty ),\frac{dt}{t})}\big)\\
&\hspace{-10cm}\leq C \left(\frac{1}{|x-y|}\int_0^\infty |\partial_u[e^{2y_1}W_u^\nu (x,y)]|du\right.\\
&\hspace{-10cm}\quad +\left.\int_0^\infty \Big(|\partial _u[\partial_{x_j}(e^{2y_1}W_u^\nu (x,y))]|+|\partial _u[\partial_{y_j}(e^{2y_1}W_u^\nu (x,y))]|\Big)du\right),\quad j=1,\ldots, n.
\end{align*}
For each $x,y\in \mathbb R^n$ and $u>0$,
\begin{equation}\label{eW}
e^{2y_1}W_u^\nu (x,y)=\frac{1}{(4\pi u)^{n/2}}e^{y_1-x_1}e^{-u}e^{-\frac{|x-y|^2}{4u}},
\end{equation}
and  
\begin{align}\label{1.4}
    \partial _u[e^{2y_1}W_u^\nu (x,y)]&=e^{2y_1}W_u^\nu (x,y)\left(-\frac{n}{2u}-1+\frac{|x-y|^2}{4u^2}\right)\nonumber\\
    &=e^{2y_1}W_u^\nu (x,y)\left(\frac{|x-y|^2-2nu-4u^2}{4u^2}\right).
\end{align}
Moreover, straightforward manipulations give, for every $ x,y\in \mathbb R^n$ and $u>0$,
\begin{equation}\label{derivx1}
\partial_{y_1}(e^{2y_1}W_u^\nu (x,y))=-\partial_{x_1}(e^{2y_1}W_u^\nu (x,y))=e^{2y_1}W_u^\nu (x,y)\Big(1+\frac{x_1-y_1}{2u}\Big),
\end{equation}
and, for every $j=2,\ldots,n$,
\begin{equation}\label{derivxj}
\partial_{y_j}(e^{2y_1}W_u^\nu (x,y))=-\partial_{x_j}(e^{2y_1}W_u^\nu (x,y))=e^{2y_1}W_u^\nu (x,y)\frac{x_j-y_j}{2u},\quad x,y\in \mathbb R^n\mbox{ and }u>0.
\end{equation}
By taking into account \eqref{1.4} we obtain, for $x,y\in \mathbb R^n$ and $u>0$,
\begin{align}\label{paruparx1}
\partial _u[\partial_{y_1}(e^{2y_1}W_u^\nu (x,y))]&=-\partial _u[\partial_{x_1}(e^{2y_1}W_u^\nu (x,y))]\nonumber\\
&=e^{2y_1}W_u^\nu (x,y)\left(\Big(\frac{|x-y|^2-2nu-4u^2}{4u^2}\Big)\Big(1+\frac{x_1-y_1}{2u}\Big)-\frac{x_1-y_1}{2u^2}\right),
\end{align}
and, when $j=2,\ldots,n$,
\begin{align}\label{paruparxj}
 \partial _u[\partial_{y_j}(e^{2y_1}W_u^\nu (x,y))]&=-\partial _u[\partial_{x_j}(e^{2y_1}W_u^\nu (x,y))\nonumber\\
 &=e^{2y_1}W_u^\nu (x,y)\left(\Big(\frac{|x-y|^2-2nu-4u^2}{4u^2}\Big)\frac{x_j-y_j}{2u}-\frac{x_j-y_j}{2u^2}\right).
\end{align}

By \eqref{1.4} (respectively, \eqref{paruparx1} and \eqref{paruparxj}) we deduce that for each $x,y\in \mathbb R^n$, the sign of $\partial _uW_u^\nu (x,y)$, $u>0$, (respectively, $\partial _u[\partial_{y_j}(e^{2y_1}W_u^\nu (x,y))]$ and $\partial _u[\partial_{x_j}(e^{2y_1}W_u^\nu (x,y))]$, $j=1,\ldots,n$) changes at most two times (respectively, three times). 

It follows that
$$
\|\varphi(x,y)\mathcal S_r (x,y,\cdot )\|_{L^2((0,\infty ),\frac{dt}{t})}\leq C\sup_{u>0}[e^{2y_1}W_u^\nu (x,y)],\quad x,y\in \mathbb R^n,
$$
and
\begin{align*}
    \sum_{j=1}^n\Big(\|\partial _{x_j}[\varphi(x,y)\mathcal S_r (x,y,\cdot )]\|_{L^2((0,\infty ),\frac{dt}{t})}+\|\partial_{y_j}[\varphi(x,y)\mathcal S_r (x,y,\cdot )]\|_{L^2((0,\infty ),\frac{dt}{t})}\Big)\\
&\hspace{-10cm}\leq C\left(\frac{1}{|x-y|}\sup_{u>0}[e^{2y_1}W_u^\nu (x,y)]\right.\\
    &\hspace{-10cm}\quad + \left.\sum_{j=1}^n\Big(\sup_{u>0}|\partial _{x_j}(e^{2y_1}W_u^\nu (x,y))|+\sup_{u>0}|\partial _{y_j}(e^{2y_1}W_u^\nu (x,y))|\Big)\right),\quad x,y\in \mathbb R^n,\,x\not=y.
\end{align*}
Now, by taking into account \eqref{eW}, \eqref{derivx1} and \eqref{derivxj}, we have that, for each $x,y\in \mathbb R^n$ and $u>0$,
\begin{equation}\label{eyW}
e^{2y_1}W_u^\nu (x,y)\leq \frac{1}{(4\pi u)^{n/2}}e^{|x-y|-u-\frac{|x-y|^2}{4u}}=\frac{C}{u^{n/2}}e^{-\frac{(2u-|x-y|)^2}{4u}},
\end{equation}
and, for every $j=1,\ldots,n$,
$$
|\partial _{x_j}(e^{2y_1}W_u^\nu (x,y))|+|\partial _{y_j}(e^{2y_1}W_u^\nu (x,y))|\leq \frac{C}{u^{n/2}}e^{-\frac{(2u-|x-y|)^2}{4u}}\Big(1+\frac{|x-y|}{u}\Big).
$$
Suppose that $x,y\in \mathbb R^n$ and $0<|x-y|<2$. We distingue three cases:

(i) If $u>|x-y|$, then $2u-|x-y|>|x-y|$ and we can write
$$
e^{2y_1}W_u^\nu (x,y)\leq \frac{C}{u^{n/2}}e^{-\frac{|x-y|^2}{4u}}\leq \frac{C}{|x-y|^n},
$$
and also, for $j=1,\ldots,n$,
\begin{align*}
    |\partial _{x_j}(e^{2y_1}W_u^\nu (x,y))|+|\partial _{y_j}(e^{2y_1}W_u^\nu (x,y))|&\leq \frac{C}{u^{n/2}}e^{-\frac{|x-y|^2}{4u}}\Big(1+\frac{|x-y|}{u}\Big)\leq \frac{C}{u^{n/2}}e^{-\frac{|x-y|^2}{4u}}\\
    &\leq \frac{C}{|x-y|^n}\leq \frac{C}{|x-y|^{n+1}}.
\end{align*}

(ii) If $0<u<|x-y|/4$, then $|x-y|-2u>|x-y|/2$ and thus
$$
e^{2y_1}W_u^\nu (x,y)\leq \frac{C}{u^{n/2}}e^{-\frac{|x-y|^2}{16u}}\leq \frac{C}{|x-y|^n},
$$
and for $j=1,\ldots,n$,
\begin{align*}
    |\partial _{x_j}(e^{2y_1}W_u^\nu (x,y))|+|\partial _{y_j}(e^{2y_1}W_u^\nu (x,y))|&\leq \frac{C}{u^{n/2}}e^{-\frac{|x-y|^2}{16u}}\Big(1+\frac{|x-y|}{u}\Big)\leq C\frac{|x-y|}{u^{n/2+1}}e^{-\frac{|x-y|^2}{16u}}\\
    &\leq \frac{C}{|x-y|^{n+1}}.
\end{align*}

(iii) If $|x-y|/4\leq u\leq |x-y|$, then
$$
e^{2y_1}W_u^\nu (x,y)\leq \frac{C}{u^{n/2}}\leq \frac{C}{|x-y|^{n/2}}\leq \frac{C}{|x-y|^n},
$$
and also, for $j=1,\ldots,n$,
$$
|\partial _{x_j}(e^{2y_1}W_u^\nu (x,y))|+|\partial _{y_j}(e^{2y_1}W_u^\nu (x,y))|\leq \frac{C}{u^{n/2}}\leq \frac{C}{|x-y|^{n/2}}\leq \frac{C}{|x-y|^{n+1}}.
$$
From these estimates (a) and (b) are established and, according to \eqref{1.3} it follows that, for each $x,y\in \mathbb R^n$, $x\not=y$,

(c) $\displaystyle \|K (x,y,\cdot )\|_{L^2((0,\infty ),\frac{dt}{t})}\leq \frac{C}{|x-y|^n}$. 

(d) $\displaystyle \| |\nabla_xK(x,y,\cdot)|\|_{L^2((0,\infty ),\frac{dt}{t})}+\| |\nabla_yK(x,y,\cdot )|\|_{L^2((0,\infty ),\frac{dt}{t})}\leq \frac{C}{|x-y|^{n+1}}$. 
\vspace{0.2cm}

Let $f\in C_c^\infty (\mathbb R^n)$ and $x\not \in {\rm supp} (f)$. We define the functions
$$
H_1(t)=A^k_{\nu ,M,t}(f\varphi(x,\cdot))(x),\quad t\in \Big(\frac{1}{N},N\Big),
$$
and 
$$
H_2(t)=\left[\int_{\mathbb R^n} f(y)K_{\nu,M,\cdot}^k (x,y)dy\right](t),\quad t\in \Big(\frac{1}{N},N\Big),
$$
where the integral is understood in the $\mathbb B_N$-Bochner sense. We are going to see that $H_1=H_2$ in $\mathbb B_N$. Let $h\in \mathbb B_N$. By using properties of Bochner integral we obtain
\begin{align*}
    \int_{1/N}^NH_2(t)h(t)\frac{dt}{t}&=\int_{\mathbb R^n}f(y)\int_{1/N}^N h(t)K_{\nu,M,t}^k (x,y)\frac{dt}{t}dy\\
    &=\int_{1/N}^Nh(t)\int_{\mathbb R^n}f(y)K_{\nu,M,t}^k (x,y)dy\frac{dt}{t}.
\end{align*}
The interchange of the order of integration is justified because, according to (c)  we have that
$$
\int_{1/N}^N|h(t)||K_{\nu,M,t}^k (x,y)|\frac{dt}{t}\leq \|h\|_{L^2((\frac{1}{N},N),\frac{dt}{t})}\|K_{\nu,M,\cdot}^k(x,y)\|_{L^2((\frac{1}{N},N),\frac{dt}{t})} \leq \frac{C}{|x-y|^n},\quad x,y\in \mathbb R^n,\,x\not=y.
$$
Consider the operator $\mathbb G_{N,{\rm loc}}: C_c^\infty (\mathbb R^n)\subseteq L^2(\mathbb R^n,\mu_\nu)\longrightarrow L^2_{\mathbb B_N}(\mathbb R^n,\mu_\nu)$ given by $f\rightarrow G_{N,{\rm loc}}(f)$. According to \eqref{1.0} we obtain
$$
\|G_{N,{\rm loc}}(f)\|_{L^2_{\mathbb B_N}(\mathbb R^n,\mu_\nu)}\leq C\|f\|_{L^2(\mathbb R^n,\mu_\nu )},\quad f\in C_c^\infty (\mathbb R^n),
$$
where $C>0$ does not depend on $N$.

By Proposition \ref{Prop2.3} we have that
$$
\|G_{N,{\rm loc}}(f)\|_{L^2_{\mathbb B_N}(\mathbb R^n,dx)}\leq C\|f\|_{L^2(\mathbb R^n,dx)},\quad f\in C_c^\infty (\mathbb R^n),
$$
with $C>0$ independent of $N$. Then, $G_{N,{\rm loc}}$ can be extended to $L^2(\mathbb R^n,dx)$ as a bounded operator from $L^2(\mathbb R^n,dx)$ into $L^2_{\mathbb B _N}(\mathbb R^n,dx)$ and
$$
\sup_{N\in \mathbb N}\|G_{N,{\rm loc}}\|_{L^2(\mathbb R^n,dx)\longrightarrow L^2_{\mathbb B_N}(\mathbb R^n,dx)}<\infty.
$$
Thus, for every $f\in C_c^\infty (\mathbb R^n)$,
$$
G_{N,{\rm loc}}(f)(x)=\int_{\mathbb R^n}f(y)K_{\nu,M,\cdot}^k (x,y) dy,\quad x\not\in {\rm supp}(f),
$$
where the integral is understood in the $\mathbb B_N$-Bochner sense.

According to (c) and (d) we get, for each $x,y\in \mathbb R^n$, $x\not=y$,
\begin{equation}\label{1.5a}
\sup_{N\in \mathbb N}\|K_{\nu,M,\cdot}^k (x,y)\|_{\mathbb B_N}\leq \frac{C}{|x-y|^n},
\end{equation}
and 
\begin{equation}\label{1.5b}
\sup_{N\in \mathbb N}\Big(\| |\nabla_x K_{\nu,M,\cdot}^k(x,y)|\|_{\mathbb B_N}+\||\nabla_y K_{\nu,M,\cdot}^k(x,y)|\|_{\mathbb B_N}\Big)\leq \frac{C}{|x-y|^{n+1}}.
\end{equation}
By using vector-valued Calder\'on-Zygmund singular integral theory we conclude that, for every $N\in \mathbb N$, the operator $G_{N,{\rm loc}}$ can be extended to $L^1(\mathbb R^n,dx)$ as a bounded operator from $L^1(\mathbb R^n,dx)$ into $L^{1,\infty}_{\mathbb B_N}(\mathbb R^n,dx)$ that we continue naming $G_{N,{\rm loc}}$ satisfying that
$$
\sup_{N\in \mathbb N}\|G_{N,{\rm loc}}\|_{L^1(\mathbb R^n,dx)\longrightarrow L^{1,\infty}_{\mathbb B_N}(\mathbb R^n,dx)}<\infty.
$$

According to Proposition \ref{Prop2.3}, for every $N\in \mathbb N$, $G_{N,{\rm loc}}$ can be extended to $L^1(\mathbb R^n,\mu_\nu)$ as a bounded operator from $L^1(\mathbb R^n,\mu_\nu)$ into $L^{1,\infty}_{\mathbb B_N}(\mathbb R^n,\mu_\nu)$. This extension, that we continue denoting by $G_{N,{\rm loc}}$, satisfies
$$
\sup_{N\in \mathbb N}\|G_{N,{\rm loc}}\|_{L^1(\mathbb R^n,\mu_\nu)\longrightarrow L^{1,\infty}_{\mathbb B_N}
(\mathbb R^n,\mu_\nu)}<\infty.
$$
Then, there exists $C>0$ such that, for every $f\in C_c^\infty (\mathbb R^n)$,
\begin{align*}
\lim_{N\rightarrow \infty}\mu_\nu \Big(\big\{x\in \mathbb R^n: \|G_{N,{\rm loc}}(f)(x)\|_{\mathbb B_N}>\lambda \big\}\Big)&\\
&\hspace{-3cm}=\mu_\nu  \Big(\big\{x\in \mathbb R^n:g_{\nu,M,{\rm loc}}^k(f)(x)>\lambda \big\}\Big)\leq \frac{C}{\lambda}\|f\|_{L^1(\mathbb R^n,\mu_\nu )},\quad \lambda >0.
\end{align*}
Suppose that $f\in L^1(\mathbb R^n,\mu_\nu)$ and choose a sequence $(f_j)_{j\in \mathbb N}\subset C_c^\infty (\mathbb R^n)$ such that $f_j\longrightarrow f$, as $j\rightarrow \infty$, in $L^1(\mathbb R^n,\mu_\nu)$. 

Let $N\in \mathbb N$, $N\geq 1$. We have that $G_{N,{\rm loc}}(f-f_\ell)\longrightarrow 0$, as $\ell\rightarrow \infty$, in $L^{1,\infty}_{\mathbb B_N}
(\mathbb R^n,\mu_\nu)$. There exists an increasing function $\phi:\mathbb N\longrightarrow \mathbb N$ such that $\|G_{N,{\rm loc}}(f-f_{\phi (\ell )})(x)\|_{\mathbb B_N}\longrightarrow 0$, as $\ell\rightarrow \infty$, for almost all $x\in \mathbb R^n$.

By taking into account \eqref{A_nuMt^k} and \eqref{1.1} it follows that
\begin{align*}
A^k_{\nu,M,t}((f-f_\ell )\varphi(x,\cdot))(x)&=\frac{1}{\Gamma (M)}t^k\partial _t^k\Big(t^{-M}\int_0^\infty e^{-u/t}W_u^\nu ((f-f_\ell)\varphi(x,\cdot))(x)u^{M-1}du\Big)\\
&\hspace{-2.5cm}=\sum_{(m_1,\ldots,m_k)\in \mathcal N_k}\sum_{j=0}^{m_1+\cdots+m_k}a_{m_1,\ldots,m_k,j}\\
&\hspace{-2.5cm}\quad \times \int_0^\infty \frac{u^{m_1+\cdots +m_k-j+M-1}}{t^{m_1+\cdots+m_k+M-j}}e^{-u/t}W_u^\nu ((f-f_\ell)\varphi (x,\cdot))(x)du,\quad x\in \mathbb R^n,\,\ell \in \mathbb N\mbox{ and }t>0.
\end{align*}
Since $W_u^\nu$, $u>0$, is contractive in $L^1(\mathbb R^n,\mu_\nu)$ we obtain
$$
\Big\|\|A^k_{\nu,M,\cdot}((f-f_\ell )\varphi (x,\cdot))(x)\|_{\mathbb B_N}\Big\|_{L^1(\mathbb R^n,\mu_\nu)}\leq C\|f-f_\ell\|_{L^1(\mathbb R^n,\mu_\nu)},\quad \ell \in \mathbb N.
$$
Thus, there exists an increasing function $\theta:\mathbb N\longrightarrow \phi(\mathbb N)$ such that, for almost all $x\in \mathbb R^n$,
$$
\lim_{j\rightarrow \infty}\|A^k_{\nu,M,\cdot}(f_{\theta (j)}\varphi (x,\cdot))(x)\|_{\mathbb B_N}=\|A^k_{\nu,M,\cdot}(f\varphi (x,\cdot))(x)\|_{\mathbb B_N}.
$$
We deduce that
$$
\|A^k_{\nu,M,\cdot}(f\varphi (x,\cdot))(x)\|_{\mathbb B_N}=\|G_{N,{\rm loc}}(f)(x)\|_{\mathbb B_N},\quad \mbox{ a.e. }x\in \mathbb R^n.
$$
A diagonal argument and letting $N\rightarrow \infty$ as above we conclude that
$$
\mu_\nu  \Big(\big\{x\in \mathbb R^n:g_{\nu,M,{\rm loc}}^k (f)(x)>\lambda \big\}\Big)\leq \frac{C}{\lambda}\|f\|_{L^1(\mathbb R^n,\mu_\nu )},\quad \lambda >0,
$$
where $C$ does not depend on $f$.

Thus, the study of the local operator is finished.

We now consider the global operator $G_{N,{\rm glob}}$. We define
\begin{equation}\label{L}
L_{\nu,M,t}^k(x,y)=\frac{1-\varphi (x,y)}{\Gamma (M)}t^k\partial _t^k\Big[t^{-M}\int_0^\infty e^{-u/t}W_{u}^\nu (x,y)u^{M-1}du\Big],\quad x,y\in \mathbb R^n \mbox{ and }t>0.
\end{equation}
By proceeding as in the local case we obtain
$$
\|L_{\nu,M,\cdot}^k(x,y)\|_{L^2((0,\infty),\frac{dt}{t})}\leq C\sup_{u>0}W_u^\nu (x,y),\quad x,y\in \mathbb R^n.
$$
According to \cite[Lemma 6]{LSjW} we have that
$$
\|L_{\nu,M,\cdot}^k (x,y)\|_{L^2((0,\infty ),\frac{dt}{t})}\leq Ce^{-(x_1+y_1)}\frac{e^{-|x-y|}}{|x-y|^{n/2}}\mathcal X_{\{|x-y|\geq 1\}}(x,y),\quad x,y\in \mathbb R^n.
$$
By using \cite[Proposition 4]{LSjW} and by letting $N\rightarrow \infty$ we deduce that $g_{\nu,M,{\rm glob}}^k$ is bounded from $L^1(\mathbb R^n,\mu_\nu)$ into $L^{1,\infty }(\mathbb R^n,\mu_\nu)$. 

Thus the proof is finished.

\section{Proof of Theorem \ref{Th1.1} for Littlewood-Paley functions $G_{M,\nu}^{k,\ell}$}
Let $f\in C_c^\infty (\mathbb R^n)$. We have that
\begin{align*}
t^{k+\ell /2}\partial _t^k D^{(\ell)}(I-t\Delta_\nu )^{-M}(f)(x)&=\frac{t^{k+\ell/2}}{\Gamma (M)}\int_0^\infty \partial _t^k \Big(e^{-u/t}\big(\frac{u}{t}\big)^{M}\Big)D^{(\ell)}W_u^\nu (f)(x)\frac{du}{u}\\
&\hspace{-3cm}=\frac{1}{\Gamma(M)}\int_0^\infty \big[s^{k+\ell/2}\partial _s^k(e^{-1/s}s^{-M})\big]_{|s=\frac{t}{u}}u^{\ell /2}D^{(\ell)}W_u^\nu (f)(x)\frac{du}{u}\\
&\hspace{-3cm}=\frac{1}{\Gamma (M)}\int_0^\infty s^{k+\ell/2}\partial _s^k(e^{-1/s}s^{-M})\Big(\frac{t}{s}\Big)^{\ell/2}D^{(\ell)}W_{t/s}^\nu (f)(x)\frac{ds}{s},\quad x\in \mathbb R^n\mbox{ and }t>0.
\end{align*}
By using Minkowski inequality and \eqref{1.1}, and taking into account that $\ell <2M$ we obtain
\begin{align*}
    G_{\nu,M}^{k,\ell}(f)(x)&\leq \int_0^\infty |s^{k+\ell/2}\partial _s^k(s^{-M}e^{-1/s})|\Big\|\Big(\frac{t}{s}\Big)^{\ell/2}D^{(\ell)}W_{t/s}^\nu (f)(x)\Big\|_{L^2((0,\infty ),\frac{dt}{t})}\frac{ds}{s}\\
    &\leq \|t^{\ell/2}D^{(\ell)}W_t^\nu (f)(x)\|_{L^2((0,\infty ),\frac{dt}{t})}\int_0^\infty |s^{k+\ell/2}\partial _s^k(s^{-M}e^{-1/s})|\frac{ds}{s}\\
    &\leq C\|t^{\ell/2}D^{(\ell)}W_t^\nu (f)(x)\|_{L^2((0,\infty ),\frac{dt}{t})}.
\end{align*}
Then, it follows that the Littlewood-Paley function $G_{\nu,M}^{k,\ell} $ defines a bounded operator on $L^p(\mathbb R^d,\mu_\nu)$, for every $1<p<\infty$ (see \cite[p. 1280]{LSj4}). 
On the other hand, 
\begin{align*}
   G_{\nu,M}^{k,\ell}(f)(x)&\leq C\int_0^\infty \big\|[s^{k+\ell/2}\partial _s^k (s^{-M}e^{-1/s})]_{|s=\frac{t}{u}}\big\|_{L^2((0,\infty ),\frac{dt}{t})}u^{\ell/2}|D^{(\ell)}W_u^\nu (f)(x)|\frac{du}{u}\\
    &\leq C\int_{\mathbb R^n}|f(y)|\int_0^\infty u^{\ell/2}|D^{(\ell)}W_u^\nu (x,y)|\frac{du}{u}d\mu_\nu(y),\quad x\in \mathbb R^n.
\end{align*}
By proceeding as in the proof of \cite[Theorem 1.2, (ii)]{LSj4} we deduce that $G_{\nu,M}^{k,\ell}$ defines a bounded operator from $L^1(\mathbb R^n,\mu _\nu)$ into $L^{1,\infty }(\mathbb R^n,\mu_\nu )$, provided that $\mathfrak q\leq 2$. Note that when $\mathfrak q\leq 1$ this assert can be deduced from \cite[Theorem 1.2, (i)]{LSj4}. In addition, also by \cite[Theorem 1.2, (i)]{LSj4} we get that when $\mathfrak q>1$, there exists $C>0$ such that, for every $f\in L(1+\ln^+L)^{\mathfrak q/2-3/4}$,
$$
\mu_\nu (\{x\in \mathbb R^n: G_{\nu, M,\ell }^k (f)(x)>\lambda \})\leq C\int_{\mathbb R^n}\frac{|f(x)|}{\lambda }\Big(1+\ln^+\frac{|f(x)|}{\lambda}\Big)^{\mathfrak q/2-3/4}d\mu_\nu (x).
$$

\section{Proof of Theorem \ref{Th1.1} for variation operators}\label{Svariation}

Let $f\in C_c^\infty (\mathbb R^n)$. From \eqref{If} we can write
$$
(I-t\Delta _\nu )^{-M}(f)(x)=\frac{1}{\Gamma (M)}\int_0^\infty e^{-v}W_{tv}(f)(x)v^{M-1}dv,\quad x\in \mathbb R^n\mbox{ and }t>0,
$$
and then,
\begin{align}\label{3.1}
A^k_{\nu ,M,t}(f)(x)&=\frac{t^{k}}{\Gamma (M)}\int_0^\infty e^{-v}\partial _t^k(W_{tv}
^\nu (f)(x))v^{M-1}dv=\frac{1}{\Gamma (M)}\int_0^\infty e^{-v}[s^k\partial _s^kW_s^\nu (f)(x)]_{|s=tv}v^{M-1}dv\nonumber\\
&=\frac{t^{-M}}{\Gamma (M)}\int_0^\infty e^{-s/t}s^k\partial _s^kW_s^\nu (f)(x)s^{M-1}ds,\quad x\in \mathbb R^n\mbox{ and }t>0.
\end{align}
The differentiation under the integral sign is justified because when $M>n/2$, for every $\ell \in \mathbb N$, we have (see \eqref{1.1})
\begin{align}\label{3.2}
\int_0^\infty e^{-v}v^{M-1}\int_{\mathbb R^n}e^{y_1-x_1}\Big|\partial _t^\ell \Big(e^{-tv}\frac{e^{-\frac{|x-y|^2}{4tv}}}{(tv)^{n/2}}\Big)\Big||f(y)|dydv\nonumber\\
&\hspace{-7cm} \leq C\int_0^\infty e^{-v}v^{M-1}\int_{\mathbb R^n}e^{y_1-x_1}\sum_{r=0}^\ell e^{-tv}v^{\ell -r}\Big|\partial_t^r\Big(\frac{e^{-\frac{|x-y|^2}{4tv}}}{(tv)^{n/2}}\Big)\Big||f(y)|dydv\nonumber\\
&\hspace{-7cm}\leq C\int_0^\infty e^{-v}v^{M+\ell -1}\int_{\mathbb R^n}e^{y_1-x_1}\sum_{r=0}^\ell e^{-tv}\frac{e^{-c\frac{|x-y|^2}{tv}}}{(tv)^{r+n/2}}|f(y)|dydv\nonumber\\
&\hspace{-7cm} \leq C\sum_{r=0}^\ell \frac{1}{t^{r+n/2}}\int_0^\infty e^{-(t+1)v}v^{M+\ell -r-n/2-1}dv<\infty,\quad x\in \mathbb R^n\mbox{ and }t>0.
\end{align}
According to \cite[Corollary 4.5]{LeMX2}, the $\rho$-variation operator $V_\rho (\{t^k\partial _t^kW_t^\nu \}_{t>0})$ is bounded on $L^p(\mathbb R^d,\mu_\nu)$, for every $1<p<\infty$. Then, by \eqref{3.1} the $\rho$-variation operator $V_\rho (\{A^k_{\nu,M,t}\}_{t>0})$ can be extended to $L^p(\mathbb R^n,\mu_\nu)$ as a bounded operator, for each $1<p<\infty$.

We are going to see that the $\rho$-variation operator $V_\rho (\{A^k_{\nu,M,t}\}_{t>0})$ defines a bounded operator from $L^1(\mathbb R^n,\mu_\nu)$ into $L^{1,\infty }(\mathbb R^n,\mu_\nu)$. Consider a function $\varphi $ as the given in section \ref{SS2} before Proposition \ref{Prop2.2} and define, for every $t>0$, 
$$
W_{t,{\rm loc}}^\nu (f)(x)=W_t^\nu (f\varphi (x,\cdot))(x),\quad x\in \mathbb R^n,
$$
and $W_{t,{\rm glob}}^\nu (f)=W_t^\nu (f)-W_{t,{\rm loc}}^\nu (f)$. 

Let 
$$
(I-t\Delta_\nu)_{\rm loc}^{-M}(f)(x)=\frac{t^{-M}}{\Gamma (M)}\int_0^\infty e^{-u/t}W_{u,{\rm loc}}^\nu (f)(x)u^{M-1}du,\quad x\in \mathbb R^n,
$$
and $(I-t\Delta_\nu)_{\rm glob}^{-M}=(I-t\Delta_\nu)^{-M}-(I-t\Delta_\nu)_{\rm loc}^{-M}$.

We define the following $\rho$-variation operators
$$
V_{\rho,{\rm loc}}(\{A^k_{\nu,M,t}\}_{t>0})=V_\rho (\{t^k\partial _t^k(I-t\Delta_\nu )^{-M}_{\rm loc}\}_{t>0}),
$$
and
$$
V_{\rho,{\rm glob}}(\{A^k_{\nu,M,t}\}_{t>0})=V_\rho (\{t^k\partial _t^k(I-t\Delta_\nu )^{-M}_{\rm glob}\}_{t>0}).
$$
It is clear that
$$
V_\rho(\{A^k_{\nu,M,t}\}_{t>0})\leq V_{\rho,{\rm loc}}(\{A^k_{\nu,M,t}\}_{t>0})+V_{\rho,{\rm glob}}(\{A^k_{\nu,M,t}\}_{t>0}).
$$
We establish first that $V_{\rho,{\rm loc}}(\{A^k_{\nu,M,t}\}_{t>0})$ defines a bounded operator from $L^1(\mathbb R^n,\mu_\nu)$ into $L^{1,\infty }(\mathbb R^n,\mu_\nu)$.

Let $f\in C_c^\infty (\mathbb R^n)$.  By \eqref{3.2} we deduce that, for every $x\in \mathbb R^n$, the function
$$
F_k(t)=A^k_{\nu,M,t}(f)(x),\quad t\in (0,\infty ),
$$
is continuous. Then, for every $x\in \mathbb R^n$,
\begin{align*}
    V_{\rho,{\rm loc}}(\{A^k_{\nu,M,t}\}_{t>0})(f)(x)&\\
    &\hspace{-4cm}=\sup_{\substack{0<t_1<\cdots<t_\ell\\\{t_1,\ldots,t_\ell\}\subset \mathbb Q,\,\ell \in \mathbb N}}\left(\sum_{j=1}^{\ell -1}|t^k\partial _t^k (I-t\Delta_\nu)^ {-M}_{\rm loc}(f)(x)_{|t=t_j}-t^k\partial _t^k (I-t\Delta_\nu)^ {-M}_{\rm loc}(f)(x)_{|t=t_{j+1}}\right)^{1/\rho }.
\end{align*}
Since the set of the finite subsets of $\mathbb Q$ is countable we conclude that the function $V_{\rho,{\rm loc}}(\{A^k_{\nu,M,t}\}_{t>0})(f)$ is measurable in $\mathbb R^n$.

Let $N\in \mathbb N$. We define the operator $V_{\rho,{\rm loc}}^N(\{A^k_{\nu,M,t}\}_{t>0})$ as $V_{\rho,{\rm loc}}(\{A^k_{\nu,M,t}\}_{t>0})$ but taking the supremum on $\{t_1,\ldots,t_\ell\}$, $\ell \in \mathbb N$, such that $1/N<t_1<\cdots <t_\ell <N$. It is clear that
$$
\lim_{N\rightarrow \infty}V_{\rho,{\rm loc}}^N(\{A^k_{\nu,M,t}\}_{t>0})(f)(x)=V_{\rho,{\rm loc}}(\{A^k_{\nu,M,t}\}_{t>0})(f)(x),\quad x\in \mathbb R^n.
$$

Let us consider the space $E_{\rho ,N}$ of all those functions $g:(\frac{1}{N},N)\longrightarrow \mathbb C$ such that
$$
\|g\|_{E_{\rho ,N}}:=\sup_{\frac{1}{N}<t_1<\cdots<t_\ell<N,\,\ell \in \mathbb N}\left(\sum_{j=1}^{\ell -1}|g(t_j)-g(t_{j+1})|^ \rho \right)^{1/\rho}<\infty .
$$
By identifying those functions that differ in a constant in $[1/N,N]$, $(E_{\rho ,N},\|\cdot\|_{E_{\rho,N}})$ is a Banach space. If $g:(0,\infty )\longrightarrow \mathbb C$ is a derivable function and $g'\in L^1((0,\infty ),dt)$ we have that
\begin{equation}\label{3.3}
    \|g\|_{E_{\rho ,N}}=\sup_{\frac{1}{N}<t_1<\cdots<t_\ell<N,\,\ell \in \mathbb N}\left(\sum_{j=1}^{\ell -1}\Big|\int_{t_j}^{t_{j+1}}g'(t)dt\big|^ \rho \right)^{1/\rho}\leq \int_{1/N}^N|g'(t)|dt\leq \int_0^\infty |g'(t)|dt.
\end{equation}
Consider $K_{\nu,M,t}^k (x,y)$, $x,y\in \mathbb R^n$ and $t>0$, as in \eqref{K}. According to \eqref{A1} we have that, for each $x,y\in \mathbb R^n$ and $t>0$, 
\begin{align*}
\partial _tK_{\nu,M,t}^k (x,y)&=\frac{e^{2y_1}}{\Gamma (M)}\varphi (x,y)(kt^{k-1}\partial _t^{k-1}+t^k\partial _t^k)\Big(t^{-M-1}\int_0^\infty e^{-u/t}\partial _uW_u^\nu (x,y)u^Mdu\Big)\nonumber\\
&=\frac{1}{t}\big(kK_{\nu,M,t}^k(x,y)+K_{\nu ,M,t}^{k+1}(x,y)\big),\quad k\geq 1,
\end{align*}
and, when $k=0$, we have that $\partial _tK_{\nu,M,t}^0(x,y)=\frac{1}{t}K_{\nu,M,t}^1(x,y)$. Thus, it follows that
\begin{equation}\label{derivK}
\partial _tK_{\nu,M,t}^k (x,y)=\frac{1}{t}(kK_{\nu,M,y}^k(x,y)+K_{\nu ,M,t}^{k+1}(x,y)),\quad k\in \mathbb N.
\end{equation}
We are going to see that
\begin{equation}\label{3.4}
\|K_{\nu,M,t}^k (x,y)\|_{E_{\rho ,N}}\leq \frac{C}{|x-y|^n},\quad x,y\in \mathbb R^n,\,x\not=y,
\end{equation}
where the constant $C>0$ does not depend on $N$.

We recall that $0\leq \varphi (x,y)\leq 1$, $x,y\in \mathbb R^n$ and that $\varphi(x,y)=0$, when $|x-y|\geq 2$. Then, by using \eqref{3.3} and taking into account \eqref{1.3} it is sufficient to see that, for every $r>0$,
\begin{equation}\label{Sellt}
\int_0^\infty |\mathcal S_r (x,y,t)|\frac{dt}{t}\leq \frac{C}{|x-y|^n},\quad x,y\in \mathbb R^n,\,0<|x-y|<2,
\end{equation}
where $\mathcal S_r (x,y,t)$, $x,y\in \mathbb R^n$ and $t>0$, is given by \eqref{Sl}.

Let $r>0$. We have that
\begin{align*}
    \int_0^\infty |\mathcal S_r (x,y,t)|\frac{dt}{t}&\leq Ce^{2y_1}\int_0^\infty u^r|\partial _uW_u^\nu (x,y)|\int_0^\infty \frac{e^{-u/t}}{t^{r+1}}dtdu\leq Ce^{2y_1}\int_0^\infty |\partial _uW_u^\nu (x,y)|du\\
    &\leq Ce^{2y_1}\sup_{u>0}|W_u^\nu (x,y)|\leq \frac{C}{|x-y|^n},\quad x,y\in \mathbb R^n,\,0<|x-y|<2.
\end{align*}
We have used the estimates established in the proof of the property (a) after \eqref{Sl}.

Now we are going to show that, for every $x,y\in \mathbb R^n$, $x\not=y$,
\begin{equation}\label{CZ2K}
\| |\nabla_x(K_{\nu,M,t}^k (x,y))|\|_{E_{\rho ,N}}+\| |\nabla_y(K_{\nu,M,t}^k (x,y))|\|_{E_{\rho ,N}}\leq \frac{C}{|x-y|^{n+1}},
\end{equation}
with $C>0$ independent of $N$. Again by considering \eqref{1.3} we only need to see that, for $r>0$, there exists $C>0$ independent of $N$ such that, for $x,y\in \mathbb R$, $x\not=y$,
\begin{equation}\label{3.5}
\| |\nabla_x[\varphi(x,y)\mathcal S_r(x,y,\cdot)]|\|_{E_{\rho ,N}}+\| |\nabla_y[\varphi(x,y)\mathcal S_r(x,y,\cdot )]|\|_{E_{\rho ,N}}\leq \frac{C}{|x-y|^{n+1}}.
\end{equation}
Fix $r>0$. For $j=1,\ldots,n$ we have that
\begin{equation}\label{dxjK}
\partial_{x_j}[\varphi (x,y)\mathcal S_r(x,y,t)]=\varphi (x,y)\partial_{x_j}\mathcal S_r(x,y,t)+\mathcal S_r(x,y,t)\partial_{x_j}\varphi (x,y), \quad x,y\in \mathbb R^n,\,t>0,
\end{equation}
and
\begin{equation}\label{dyjK}
\partial_{y_j}[\varphi (x,y)\mathcal S_r(x,y,t)]=\varphi (x,y)\mathcal S_r(x,y,t)+\mathcal S_r(x,y,t)\partial_{y_j}\varphi (x,y), \quad x,y\in \mathbb R^n,\,t>0.
\end{equation}
According to the properties of the function $\varphi$ and \eqref{Sellt}, in order to prove \eqref{3.5} it is sufficient to establish that, for every $j=1,\ldots, n$,
$$
\|\partial_{x_j}\mathcal S_r(x,y,\cdot)\|_{E_{\rho ,N}}+\|\partial_{y_j}\mathcal S_r(x,y,\cdot)\|_{E_{\rho ,N}}\leq \frac{C}{|x-y|^{n+1}},\quad x,y\in \mathbb R^n,\,0<|x-y|<2,
$$
where $C>0$ does not depend on $N$. Taking into account \eqref{3.3} this will be proved by showing that, for each $j=1,\ldots, n$,
\begin{equation}\label{dtdxjdyjS}
    \int_0^\infty \Big(|\partial _t\partial _{x_j}\mathcal S_r (x,y,t)|+|\partial _t\partial _{y_j}\mathcal S_r(x,y,t)|\Big)dt\leq \frac{C}{|x-y|^{n+1}},\quad x,y\in \mathbb R^n,\,0<|x-y|<2,
\end{equation}
with $C>0$ independent of $N$.

Since, for certain $C>0$,
$$
\int_0^\infty \big|\partial _t\big[\big(\frac{u}{t}\big)^r e^{-u/t}\big]\Big|dt\leq C\int_0^\infty \Big(\frac{u^r}{t^{r +1}}+\frac{u^{r +1}}{t^{r+2}}\Big)e^{-u/t}dt=C, \quad u>0,
$$
it follows that
\begin{align*}
    \int_0^\infty \Big(|\partial _t\partial _{x_j}\mathcal S_r (x,y,t)|+|\partial _t\partial _{y_j}\mathcal S_r (x,y,t)|\Big)dt&\\
    &\hspace{-4cm}\leq C\int_0^\infty (|\partial _u\partial_{x_j}(e^{2y_1}W_u^\nu (x,y))|+|\partial _u\partial_{y_j}(e^{2y_1}W_u^\nu (x,y))|)du,\quad x,y\in \mathbb R^n.
\end{align*}
 By following the proof of property (b) after \eqref{1.3} we get \eqref{dtdxjdyjS} with $C>0$ independent of $N$.

Let $f\in C_c^\infty(\mathbb R^n)$. Fix $x\in \mathbb R^n$, define
$$
F_x(y)(t)=K_{\nu,M,t}^k(x,y)f(y),\quad y\in \mathbb R^n\mbox{ and }t>0,
$$
and consider the function $G:\mathbb R^n\longrightarrow E_{\rho ,N}$ given by $G(y)=F_x(y)$, $y\in \mathbb R^n$. Note that $G$ is a continuous function. Indeed, let $y_0\in \mathbb R^n$. According to \eqref{3.3} we have that
$$
\|F_x(y)-F_x(y_0)\|_{E_{\rho ,N}}\leq \int_{1/N}^N\big|\partial _t\big[K_{\nu,M,t}^k(x,y)f(y)-K_{\nu,M,t}^k(x,y_0)f(y_0)\big]\big|dt.
$$
The function
$$
(t,y)\in \Big[\frac{1}{N},N\Big]\times B(y_0,1)\rightarrow \partial _t\big[K_t(x,y)f(y)\big],
$$
is uniformly continuous in $[\frac{1}{N},N]\times B(y_0,1)$. Then,
$$
\|F_x(y)-F_x(y_0)\|_{E_\rho }\longrightarrow 0,\mbox{ as }y\rightarrow y_0,
$$
that is, $G$ is continuous and, consequently, $G$ is strongly measurable.

By \eqref{3.4} it follows that
$$
\int_{\mathbb R^n}\|G(y)\|_{E_{\rho ,N}}dy<\infty,
$$
provided that $x\not\in \supp f$.

Let $x\not\in \supp f$. Define $g(t)=A^k_{\nu,M,t}(f\varphi (x,\cdot ))(x)$, $t>0$. We have that
$$
g(t)=\int_{\mathbb R^n}K_{\nu,M,t}^k (x,y)f(y)dy,\quad t\in (0,\infty).
$$
Note that from \eqref{3.4}, $g\in E_{\rho ,N}$ and $\|g\|_{E_{\rho ,N}}\leq C\int_{\mathbb R^n}\frac{|f(y)|}{|x-y|^n}dy$.

On the other hand, let us consider the function
$$
h=\int_{\mathbb R^n}K_{\nu ,M,\cdot}^k(x,y)f(y)dy,
$$
where the integral is understood in the $\mathbb B_N$-Bochner sense.

Let $a\in (\frac{1}{N}
,N)\setminus \{1\}$ and consider $L_a(\alpha)=\alpha (a)-\alpha (1)$, $\alpha \in E_{\rho ,N}$. It is clear that $L_a\in E_{\rho ,N}'$, the dual space of $E_{\rho ,N}$. According to the properties of the Bochner integral we have that
$$
L_a(h)=\int_{\mathbb R^n}K_{\nu,M,a}^k (x,y)f(y)dy-\int_{\mathbb R^n}K_{\nu,M,1}^k (x,y)f(y)dy=g(a)-g(1).
$$
It follows that $h(a)-g(a)=h(1)-g(1)$ and we can conclude that $h=g$ in $E_{\rho ,N}$.

We now apply the Calder\'on-Zygmund theory for vector-valued singular integrals (\!\! \cite{RubRT}). Since the variation operator $V_\rho (\{A^k_{\nu,M,t}\}_{t>0})$ is bounded on $L^p(\mathbb R^n,\mu_\nu)$, for every $1<p<\infty$, by Proposition \ref{Prop2.2} we deduce that the operator $V_{\rho,{\rm loc}} (\{A^k_{\nu,M,t}\}_{t>0})$ is also bounded on $L^p(\mathbb R^n,\mu_\nu)$, for every $1<p<\infty$. Hence, we have that the operator $T_{\rm loc}$ defined by
$$
T_{\rm loc}(f)(x)=A^k_{\nu,M,\cdot}(f\varphi (x,\cdot ))(x),\quad x\in \mathbb R^n,
$$
defines a bounded operator from $L^p(\mathbb R^n,dx)$ into $L^p_{E_{\rho ,N}}(\mathbb R^n,dx)$, for every $1<p<\infty$. We also have that, for every $f\in C_c^\infty (\mathbb R^n)$,
$$
T_{\rm loc}(f)(x)=\int_{\mathbb R^n}K_{M,\cdot }^\nu (x,y)f(y)dy,\quad x\not\in \supp f.
$$
Then, by \eqref{3.4} and \eqref{CZ2K} we deduce that $T_{\rm loc}$ can be extended to $L^1(\mathbb R^n,dx)$ as a bounded operator from $L^1(\mathbb R^n,dx)$ into $L^{1,\infty }_{E_{\rho ,N}}(\mathbb R^n,dx)$. From Proposition \ref{Prop2.3} it follows that $T_{\rm loc}$ can be extended to $L^1(\mathbb R^n,\mu_\nu)$ as a bounded operator from $L^1(\mathbb R^n,\mu_\nu)$ into $L^{1,\infty }_{E_{\rho ,N}}(\mathbb R^n,\mu_\nu)$. In addition we have that
$$
\sup_{N\in \mathbb N}\|T_{\rm loc}\|_{L^1(\mathbb R^n,\mu _\nu)\longrightarrow L^{1,\infty}_{E_{\rho ,N}}(\mathbb R^n,\mu_\nu)}<\infty.
$$
By taking limit as $N\rightarrow \infty$ we get, for every $f\in C_c^\infty (\mathbb R^n)$,
$$
\mu_\nu (\{x\in \mathbb R^n: V_{\rho , {\rm loc}}(\{A^k_{\nu,M,t}\}_{t>0})(f)(x)>\lambda \})\leq \frac{C}{\lambda}\|f\|_{L^1(\mathbb R^n,\mu_\nu )},\quad \lambda >0.
$$
Let us deal now with $V_{\rho , {\rm glob}}(\{A^k_{\nu,M,t}\}_{t>0})$. Let $f\in C_c^\infty (\mathbb R^n)$. As in \eqref{3.3} we have that 
\begin{align*}
V_{\rho , {\rm glob}}(\{A^k_{\nu,M,t}\}_{t>0})(f)(x)&\leq \int_0^\infty |\partial _t[A^k_{\nu,M,t}((1-\varphi (x,\cdot ))f)(x)]|dt\\
&\leq C\int_{\mathbb{R}^n}|f(y)|e^{2y_1}\int_0^\infty |\partial _t L_{\nu ,M,t}^k(x,y)|dtdy,\quad x\in \mathbb R^n,
\end{align*}
where $L_{\nu,M,t}^k(x,y)$, $x,y\in \mathbb R^n$, $t>0$ is given by \eqref{L}.

As in the study of $g_{\nu,M,{\rm glob}}^k$ we can see that
$$
\int_0^\infty |\partial _tL_{\nu ,M,t}^k(x,y)|dt\leq C\sup_{u>0}|W_u^\nu (x,y)|,\quad x,y\in \mathbb R^n\mbox{ and }|x-y|\geq 1.
$$
Then, again by \cite[Lemma 6]{LSjW} we can write
$$
V_{\rho ,{\rm glob}}(\{A^k_{\nu,M,t}\}_{t>0})(f)(x)\leq C\int_{|x-y|\geq 1}e^{-(x_1+y_1)}\frac{e^{-|x-y|}}{|x-y|^{n/2}}f(y)d\mu_\nu (y),\quad x\in \mathbb R^n,
$$
and according to \cite[Proposition 4]{LSjW} we conclude that $V_{\rho ,{\rm glob}}(\{A^k_{\nu,M,t})\}_{t>0}$ defines a bounded operator from $L^1(\mathbb R^n,\mu_\nu)$ into $L^{1,\infty }(\mathbb R^n,\mu_\nu )$.

Thus, the proof is finished.

\bibliographystyle{acm}

\end{document}